\documentclass{article}

\usepackage{amsmath,amssymb,latexsym,amsthm,mathrsfs,gensymb,marginnote,authblk}
\usepackage{tikz,color}

\newtheorem{theo}{Theorem}[section]

\newtheorem{lem}[theo]{Lemma}
\newtheorem{cor}[theo]{Corollary}

\newcommand{\ra}{\rightarrow}
\theoremstyle{definition}
\newtheorem{defin}[theo]{Definition}
\newtheorem{exa}[theo]{Example}

\theoremstyle{remark}
\newtheorem{rem}[theo]{Remark}

\title{Gambler's ruin estimates on finite inner uniform domains}

\author[1]{Persi Diaconis}
\author[2]{Kelsey Houston-Edwards\thanks{kedwards@olin.edu}}
\author[3]{Laurent Saloff-Coste\thanks{lsc@math.cornell.edu}}
\affil[1]{\textit{\small Departments of Mathematics and Statistics, Stanford University}}
\affil[2]{\textit{\small The Olin College of Engineering}}
\affil[3]{\textit{\small Department of Mathematics, Cornell University}}

\begin{document}

\maketitle

\begin{abstract} 
Gambler's ruin estimates can be viewed as harmonic measure estimates for finite Markov chains which are absorbed (or killed) at boundary points. We relate such estimates to properties of the underlying chain and its Doob transform. Precisely, we show that gambler's ruin estimates reduce to a good understanding of the Perron-Frobenius eigenfunction and eigenvalue whenever the underlying chain and its Doob transform  are Harnack Markov chains.  Finite inner-uniform domains (say, in the square grid $\mathbb Z^n$)
provide a large class of examples where these ideas apply and lead to detailed estimates.
In general, understanding the behavior of the Perron-Frobenius eigenfunction remains a challenge.
\end{abstract}

\section{Introduction}

Two players are involved in a simple fair game that is repeated, independently, many times. Assume that the total amount of money involved is $N$ and that  we follow $X_t$, the amount of money that player $A$ holds at time $t$.  We can view $X_t$ as performing a simple random walk on $\{0,1,\dots, N\}$ with absorbing boundary condition at both ends. The classical gambler's ruin problem asks for the computation of  the probability that $A$ wins (i.e.,  there is an $t$ such that $X_t=N$ and $X_k\neq 0$ for $0\le k\le t$) given that $X_0=x$. Call this probability $u(x)$. Then,  $u(0)=0$, $u(N)=1$, and, for $0<x
<N$, $u(x)=\frac{1}{2}(u(x-1)+u(x+1))$.  In a different language, $u$ is the solution of the discrete Dirichlet problem on $\{0,\dots,N\}$
$$\left\{\begin{array}{cl} \Delta u=0 & \mbox{ on } U=\{1,\dots, N-1\},\\
u=\phi &\mbox{ on } \partial U=\{0,N\},\end{array}\right.$$
 with boundary function $\phi(0)=0$ and $\phi(N)=1$, and Laplacian 
 $$\Delta u(x)=u(x)-\frac{1}{2}(u(x-1)+u(x+1)).$$
 Because the only harmonic functions on the discrete line are the affine functions it follows immediately that $u(x)= x/N.$ For example, if you have \$1 and your opponent has \$99, the chance that you eventually win all the money is 1/100 (see \cite[Chapter 14]{feller} for an inspirational development). This naturally leads to the question: how should the gambler's ruin problem be developed with more players?

 \begin{figure}[h] 
\begin{center}
\begin{tikzpicture}[scale=.2]

\foreach \x in {1,2,3,4,5,6,7,8,9,10,11,12,13,14,15,16,17,18,19,20}
{\draw  (\x,0) -- (0,\x);
\draw (\x,0) -- (\x,20-\x);
\draw (0,\x)--(20-\x,\x);
}
\draw (0,0)--(0,20);
\draw (0,0)-- (20,0);

\foreach \x in {1,2,3,4,5,6,7,8,9,10,11,12,13,14,15,16,17,18,19}
{\draw [fill, blue] (\x,0) circle [radius=.2];\draw  [fill, blue](0,\x)  circle [radius=.2];}
\foreach \x in {1,2,3,4,5,6,7,8,9,10,11,12,13,14,15,16,17,18,19}{\draw [fill, blue] (20-\x,\x) circle [radius=.2];}

\end{tikzpicture}
\caption{The gambler's ruin problem with $3$ players} \label{GR2}

\end{center}
\end{figure}
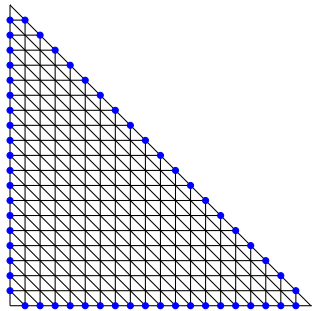

 Thomas Cover in~\cite{Cover1987} gives a multi-player version of the gambler's ruin problem using Brownian motion. It is solved using conformal maps in the $3$-player (i.e., 2-dimensional) case in a short note of Bruce Hajek \cite{Hajek1987} that appears in the same volume as Cover's article. (For another description of 3-player gambler's ruin, see~\cite{ferguson}.) The discrete $3$-player version can be described as follows. Call the players $A,B,C$.  Let $N$ be the total amount of money in the game and $X_*$  be the amount of money that  player $*$ has at a given time so that $X_A+X_B+X_C=N$. At each stage, a pair of players is chosen uniformly at random; then these two players play a fair game and exchange one dollar
 according to the outcome of the game. Standard martingale arguments show that the chance that player $A$,$B$ or $C$ winds up with all the money (given that they start out at $x_1$,$x_2$, and $x_3$) is, respectively, $x_1/N$, $x_2/N$,$x_3/N$. Starting at $N/4$,$N/4$,$N/2$, Ferguson~\cite{ferguson} shows that the chance that $C$ is the first eliminated is asymptotically 0.1421... We consider what happens the first time one of the players is eliminated. How does the money divide up among the remaining two players and how does this depend on the starting position?
 
 From this description it follows that the pair $(X_A, X_B)$ evolves  on 
 $$U=\{(x_1,x_2): 0< x_1,\;0 < x_2,\;x_1+x_2< N\},$$
 with
 \begin{eqnarray*} \partial U&=&\{(x_1,x_2):  x_1=0, 0< x_2< N\}\\&& \bigcup \;\{(x_1,x_2):  x_2=0, 0< x_1<N\}\\&&\bigcup\; \{(x_1,x_2): 0< x_1,0< x_2, x_1+x_2=N\},\end{eqnarray*} according to a Markov kernel given by
  $$K((x_1,x_2),(y_1,y_2))= \left\{\begin{array}{cl} 1/6  & \mbox{ if }  |x_1-y_1|+|x_2-y_2|=1,\\
 1/6 & \mbox{  if  }  x_1-y_1=y_2-x_2=\pm 1,\\
 0& \mbox{ otherwise,}\end{array}\right.$$
 for  pairs $(x_1,x_2)\in U,\; (y_1,y_2)\in U\cup \partial U$. Here, we imagine that this Markov chain starts somewhere in $U$, say at $(x_A,x_B)$, and runs until it first reaches a point on $\partial U$. We are interested in the probability that the exit point is $(y_A,y_B)$ given the starting point  $(x_A,x_B)$.  Contrary to the $1$-dimensional case, there is no easy closed form formula for this problem in dimension $2$ (much less  in dimension higher than $2$ and other variants). Our results, which give two-sided estimates for this problem, are developed in Example~\ref{exa-GR} and summarized in formula~\eqref{GRP}.

These examples are part of a much larger theory known under the complementary names of first passage probabilities, survival probabilities and absorption problems. In the context of classical diffusion processes, this is also related to the study of harmonic measure (see Definition~\ref{defn-harmonic measure}). See \cite{CMSM,Pinsky,Redner} among other basic relevant references.

Let us now abstract the original problem as follows. Instead of a discrete line or triangle, our new setting will be a weighted graph
$(\mathfrak X,\mathfrak E,\pi,\mu)$ where
\begin{itemize}
    \item the set $\mathfrak X$ of vertices is finite or countable,
    \item the set $\mathfrak E$ of edges consists of pairs of vertices, (i.e., subsets of $\mathfrak X$ containing exactly two elements) such that each vertex has finite degree (i.e., it belongs to only finitely many pairs in $\mathfrak E$) and the graph is connected (i.e., there is a path in $\mathfrak E$ connecting any two pairs of vertices)
    \item the function $\pi:\mathfrak X\ra (0,\infty)$ is a positive weight on vertices, and
    \item the function $\mu: \mathfrak E\ra (0,\infty)$ is a positive weight on edges, $\{x,y\}\mapsto \mu_{xy},$ with the property that
\begin{equation}\label{cond-pimu}
\sum_y \mu_{xy}\le \pi(x).\end{equation} \end{itemize}  

It is useful
to extend $\mu$ to the set of all pairs of vertices by setting $\mu_{xy}=0$ when $\{x,y\}\not\in \mathfrak E$. Two vertices $x,y$ satisfying $\{x,y\}\in \mathfrak E$ are called neighbors, which we denote $x \sim y$.  The edge set $\mathfrak E$ induces a distance function  $(x,y)\mapsto d(x,y)$ on $\mathfrak X$. The distance $d(x,y)$ between $x$ and $y$ is the minimal number of edges that have to be crossed to go from $x$ to $y$. We assume throughout that $d(x,y)$ is finite for all pairs of points $x,y\in \mathfrak E$.

This data also  induces a Markov kernel $K=K_{\pi,\mu}$  defined as follows
\begin{equation}\label{def-K}
K(x,y)=\left\{\begin{array}{cl}  \mu_{xy}/\pi(x) & \mbox{ for } y\neq x,\\
1-\left(\sum_{y}\mu_{xy}/\pi(x)\right)& \mbox{ for } y=x.\end{array} \right.
\end{equation}
Note that the pair $(K,\pi)$ is reversible.
 The associated  Laplacian is the operator $\Delta=I-K$ so that
 $$\Delta u(x)=u(x)-\sum_y K(x,y)u(y).$$
 
 Let $U$ be a finite subset of $\mathfrak X$ with the property that any two points $x,y$ in $U$
 can be connected in $U$ by a discrete path, that is, a finite sequence ${(x_0,\dots,x_k)\in U^k}$ with $x_0=x, x_k=y$ and $\{x_i,x_{i+1}\}\in \mathfrak E$, $0\le i\le k-1$.   We call such a subset a finite domain in $(\mathfrak X,\mathfrak E)$. Let $\partial U$ (the exterior boundary of $U$) be the set of vertices in $\mathfrak X\setminus U$ which have at least one neighbor in $U$.
 \begin{defin}[Inner distance] \label{def-dboundary} The smallest integer $k$ for which such a path exists for given $x,y\in U$ is denoted by $d_U(x,y)$. It is is the {\em inner distance} between $x$ and $y$ in $U$. For $x\in U$ and $y\in \partial U$, we set
$$d_U(x,y)= \min\{1+d_U(x,z): z\in U, \{z,y\}\in \mathfrak E\}.$$
 \end{defin}
 
 Let $(X_t)_{t\ge 0}$ denote the Markov chain driven by the Markov kernel $K$, starting from an initial random position $X_0$ in $U$.  This is often called a weighted random walk on the graph $(\mathfrak X,\mathfrak E)$ because, at each step, the walker either stays put or moves from its current position to one of the neighbors according to the kernel $K$.
 
  Let $\tau_U$ be the stopping time 
 $$\tau_U=\inf\{t: X_t\not\in U\}.$$
 Because the chain takes steps of distance at most 1, it must exit $U$ on the boundary (i.e., $X_{\tau_U}\in \partial U$). 
 
 \begin{defin}[Harmonic measure]\label{defn-harmonic measure} Because $X_{\tau_U} \in \partial U$, it make sense to ask for the computation of 
 $$P(x,y)=P_U(x,y) =\mathbf P(X_{\tau_U}=y|X_0=x),$$
 for $ x\in U , y\in \partial U.$
 As a function of $y$, $P(x,y)$ is called the harmonic measure (and as a function of $(x,y)$, it is also known as the Poisson kernel).
 \end{defin}
 
 The notation $P$ is used here in reference to the classical Poisson kernel in the ball of radius $r$ around the origin in $\mathbb R^n$,
 $$ P(x,\zeta)= \frac{r^2-\|x\|^2}{\omega_{n-1} r\|x-\zeta\|^n},\;\;\;x\in B_r=\{z: \|z\|< r, \zeta\in S_r=\{z:\|z\|=r\}, $$
 where $ \|z\|^2=\|(z_1,\dots,z_n)\|^2=\sum_1^nz_i^2$.
 In Euclidean space, the Poisson kernel solves the Dirichlet problem ($\Delta = -\sum_1^n\frac{\partial^2}{\partial x_i^2}$)
 $$\left\{\begin{array}{cl} \Delta u=0 & \mbox{ in } B_r,\\
 u=\phi  & \mbox{ on }  S_r=\partial B_r,\end{array}\right.$$
 in the form $$u(x)=\int_{S_r}P(x,\zeta)\phi(\zeta)d\zeta$$
 where $d\zeta$ is the  $n-1$-surface measure on $S_r$.

Similarly, the kernel $P=P_U$ on $U\times \partial U$ yields the solution of  the discrete Dirichlet problem  
 $$\left\{\begin{array}{cl} \Delta u=0 & \mbox{ in } U\\
 u=\phi  & \mbox{ on } \partial U,\end{array}\right.$$
 in the form $$u(x)=\sum_{y\in \partial U} P_U(x,y)\phi(y)=E_x(\phi(X_{\tau_U})). $$
  Observing that
 $$P_U(x,y)= E_x(\mathbf 1_{\{y\}}(X_{\tau_U}))= \mathbf P(X_{\tau_U}=y|X_0=x),$$
 we are also interested  in  understanding the quantity
 $$P_U(t,x,y)=  \mathbf P(X_{\tau_U}=y \mbox{ and } \tau_U\le t |X_0=x).$$
 
  The goal of this work is to obtain meaningful quantitative estimates for the Poisson kernel and related quantities in the weighted graph context described earlier and under strong   hypotheses on (a) the underlying weighted graph $(\mathfrak X,\mathfrak E,\pi,\mu)$ and (b) the finite domain $U\subset 
 \mathfrak X$.  The hypotheses we require are satisfied for a rich variety of interesting cases. As a test question, consider the problem of giving two-sided estimates (with upper and lower bounds differing only by a multiplicative constant) which hold uniformly for  $(x,y)\in U\times \partial U$  for the discrete Poisson kernel of a lazy simple random walk on $\mathbb Z^n$, $n\ge 1$, when $U=B(o,r)$
 is the graph ball of radius $r$ centered at the origin $o$ in $\mathbb Z^n$. For $n=1$, this is essentially the gambler's ruin problem. 
 
 Various other gambling schemes can be interpreted as random walks on polytopes with different boundaries. For example,~\cite{kmet-petkovsek} treats two gamblers with $n$ kinds of currency as a $n$-dimensional random walk---at each stage, a type of currency is chosen uniformly and then a flipped coin determines the transfer of one unit of currency.
 
 We now give a brief summary of the structure of this article. Section~\ref{sec-BC} introduces basic computations, including Poisson kernels and Green's functions. Section~\ref{sec-ST} discusses the difficulty of trying to solve these types of problems using spectral methods, even when all eigenfunctions are available. Section~\ref{sec-DT} introduces the Doob transform which changes absorbtion problems into ergodic problems. Section~\ref{sec-Harnack} gives the main new results. We introduce the notions of Harnack Markov chains and graphs, which allows us to treat the three-dimensional gambler's ruin starting ``in the middle'' in Example~\ref{exa-GR}. Section~\ref{sec-inner-uniform} specializes to nice domains (inner-uniform domains) where the results of the authors' previous paper~\cite{DHSZ}, \emph{Analytic-geometric methods for finite Markov chains with applications to quasi-stationarity}, can be harnessed. This allows uniform estimates for all starting states, in particular for the three-player gambler's ruin problem.
 
 \section{Basic computations} \label{sec-BC}
 Let us fix a weighted graph $(\mathfrak X,\mathfrak E,\pi,\mu)$ satisfying (\ref{cond-pimu}) and the associated Markov kernel $K$ defined at (\ref{def-K}) as described in the introduction. Let us also fix a finite domain $U$ and set
 $$K_U(x,y)=K(x,y)\mathbf 1_U(x)\mathbf 1_U(y).$$
 Assuming that $\partial U$ is not empty, this is a sub-Markovian kernel in the sense that $\sum_{y: y\sim x}K_U(x,y)\le 1$ for all $x \in U$ and $\sum_{y: y\sim x}K_U(x,y)<1$
 at any point $x\in U$ which has a neighbor in $\partial U$.  For any point $y\in \partial U$, define
$$\nu_U(y)=\{x\in U: \{x,y\}\in \mathfrak E\},$$
to be the set of neighbors of $y$ in $U$.
For any $x,z\in \mathfrak X$, set
 \begin{equation} \label{def-GU}G_U(x,z) = \sum_{t=0}^\infty K_U^t(x,z).\end{equation}
 
\begin{theo} For $x\in U$ and $y\in \partial U$, the Poisson kernel $P_U(x,y)$ is given by
$$P_U(x,y)= \sum_{z\in \nu_U(y)}  G_U(x,z) K(z,y)$$ 
Moreover, we have
$$P_U(t,x,y)= \sum_{z\in \nu_U(y)}  \sum_{\ell=0}^{t-1}K^{\ell}_U(x,z) K(z,y).$$ 
\end{theo}
\begin{proof} If we start at $x\in U$, in order to exit $U$ at $y$ at time $\tau_U=\ell+1$, we need to reach  a neighbor $z$ of $y$ at time $\ell$ while staying in $U$ at all earlier times and then take a last step to $y$. The probability for that is $$ \sum_{z\in \nu_U(y)} K^{\ell}_U(x,z) K(z,y).$$ 
\end{proof}

For later purposes, it is useful to restate the theorem above using slightly different notation.   First, we equip $U$ with the  measure $\pi|_U$, the restriction of the measure $\pi$ to $U$. Note that $\pi|_U$ is not normalized. The kernel $K_U$ satisfies the (so-called detailed balance) condition
$$k_U(x,y) :=K_U(x,y)/\pi|_U(y)=K_U(y,x)/\pi|_U(x).$$
The iterated kernel $k^t_U$  is the kernel of the sub-Markovian operator $$K^t_Uf(x)=\sum_yK^t_U(x,y)f(y)=\sum _y k^t_U(x,y)f(y)\pi|_U(y)$$ with respect to the measure $\pi|_U$.  Similarly, we set 
$$g_U(x,y):=G_U(x,y)/\pi(y).$$
The detailed balance condition captures the fact that $K^t_U$ is a discrete semigroup of selfadjoint operators on $L^2(U,\pi|_U)$.

Next we introduce the natural  measure on the boundary $\partial U$, $ \pi|_{\partial U}$, the restriction of $\pi$ to $\partial U$.  It simplifies notation greatly to drop the 
reference to $U$ and $\partial U$ and write $\pi|_U=\pi$, $\pi|_{\partial U}=\pi$ unless the context requires the use of the subscripts. 
For any function $f$ in $U\cup \partial U$ and point $y \in \partial U$, we define the interior normal derivative of $f$ at $y$  by
\begin{equation}
\label{eq-norm-der}
\frac{\partial f}{\partial \vec{\nu}_U}(y)= \sum _{x\in U: x\sim y} (f(x)-f(y)) \frac{\mu_{xy}}{\pi(y)}.
\end{equation}
 Now, for each $x\in U$, we view $P(x,\cdot)$ as a probability measure on  $\partial U$ and express the density $p_U(x,\cdot)$ of this probability measure with respect to the reference measure $\pi$ on the boundary, so that $p_U(x,y) = P_U(x,y)/\pi(y)$. Similarly, we set $p_U(t,x,y)=P_U(t,x,y)/\pi(y)$. 
 
 \begin{theo} \label{th-Pab1}
 For $x\in U$ and $y\in \partial U$, the Poisson kernel $P_U(x,y)$ is given by  $P_U(x,y)=p_u(x,y)\pi(y)$ with
$$p_U(x,y)= \frac{\partial_y g_U(x,y)}{\partial \vec{\nu}_U}=\sum_{t=0}^\infty\frac{\partial_y k^t_U(x,y)}{\partial \vec{\nu}_U} .$$ 
Similarly, $P_U(t,x,y)=p_U(t,x,y)\pi(y)$ with
$$p_U(t,x,y)=   \sum_{\ell=0}^{t-1} \frac{\partial_y k^\ell_U(x,y)}{\partial \vec{\nu}_U} .$$ 
\end{theo}

 A key reason that these formulas are useful is the fact that, because the functions $g_U(x,\cdot)$ and $k^t_U(x,\cdot)$ vanish at the boundary, the ``normal interior derivatives''   
 $\frac{\partial_y g_U(x,y)}{\partial \vec{\nu}_U} $ and $\frac{\partial_y k^t_U(x,y)}{\partial \vec{\nu}_U}$  are actually (weighted) finite sums of the positive values of the  relevant functions,  $g_U(x,\cdot)$ and $k^t_U(x,\cdot)$, over those neighbors of $y$ that are in $U$, i.e.,
 $$\frac{\partial_y g_U(x,y)}{\partial \vec{\nu}_U} = \sum_{z \in U: z \sim y} g_U(x,z)\frac{\mu_{yz}}{\pi(y)}$$
and similarly for $\frac{\partial_y k^t_U(x,y)}{\partial \vec{\nu}_U}$. This means that any two sided estimates on the functions $g_U,k^t_U$ themselves automatically induce two sided estimates for these ``normal interior derivatives'' for the Poisson kernel.
 
 \section{Spectral theory}\label{sec-ST}
 Unfortunately, it not easy to estimate the functions $k^t_U$ and $g_U$.  It is tempting to appeal to spectral theory in this context. The sub-Markovian operator $K_U$ is selfadjoint on $L^2(U,\pi)$ with finite spectrum $\beta_{U,i}$ and associated real eigenfunctions $\phi_{U,i}$. For simplicity, when the domain $U$ is obvious, we write $$\beta_i=\beta_{U,i}, \;\;\ \phi_i=\phi_{U,i}, \;\;(\mbox{for } 0\le i\le |U|-1).$$ We can assume the eigenvalues are ordered
 $$-1\le \beta_{|U|-1}\le \beta_{|U|-2}\le \dots\le \beta_1\le \beta_0\le 1.$$
When $\partial U\neq \emptyset$, the Perron-Frobenius theorem  asserts that $$0< \beta_0<1,\;\;\beta_{|U|-1}\ge -\beta_0,\;\; |\beta_i|< \beta_0,\;\; (\mbox{for } i=1,\dots,|U|-2),$$
and we can choose $\phi_0>0$. Moreover, $\beta_0=-\beta_{|U|-1}$ if and only if the subgraph $(U,\mathfrak E_U)$ of $(\mathfrak X,\mathfrak E)$ is bipartite  and $\sum_{y\sim x}\mu_{xy}=\pi(x) $ for all $x\in U$.  We will normalize all the eigenfunctions by $\pi|_U(|\phi_i|^2)=1$, making them unit vectors in $L^2(U,\pi)$. Note that, by convention, $\phi_i \equiv 0$ in $\mathfrak{X} \setminus U$, so we can equivalently write that $\pi(|\phi_i|^2)=1$.

This gives
\begin{equation} \label{eq-k-eigen}
 k_U^t(x,y)=\sum _{i=0}^{|U|-1} \beta_i^t\phi_i(x)\phi_i(y),
\end{equation}
and
\begin{equation} \label{eq-g-eigen}
  g_U(x,y)= \sum_{i=0}^{|U|-1} (1-\beta_i)^{-1} \phi_i(x)\phi_i(y).  
\end{equation}
Assuming for simplicity that $\beta_0>|\beta_{|U|-1}|$, the first formula yields the familiar asymptotic  
$$k^t_U(x,y)\sim \beta_0^t\phi_0(x)\phi_0(y).$$ The second formula yields almost nothing. The easy fact that $g_U(x,y)$ is positive is not visible from it, even in cases when the eigenvalues and eigenfunctions are known explicitly.

\begin{exa} \label{exa-box1}
In $\mathbb Z^2$, let $\pi$ be a uniform vertex weight (i.e., $\pi(x)\equiv 1$) and set edge weights $\mu_{xy}=1/8$ when $x\sim y$, $x,y\in \mathbb Z^2$. It follows that $K(x,y)$ at (\ref{def-K}) is the Markov kernel of the lazy random walk on $\mathbb Z^2$ (this walk stays put with probability $1/2$ or moves to one of the four neighbors chosen uniformly at random with probability $1/8$). Let $U \subseteq \mathbb Z^2$ be the box $\{-N,\dots,N\}^2$. 
Because of the product structure of both the set $U$ and the kernel $K_U$, we can write down explicitly the spectrum and eigenfunctions.
The eigenfunctions  are  the products
$$ \phi_{a,b}(x_1,x_2)= \frac{1}{N+1}\psi_a(x_1)\psi_b(x_2)  $$ where 
$$ \psi_a(k)= \begin{cases}{cl} \cos\frac{a k\pi }{2(N+1)} & \mbox{ if }  a=1,3, \dots ,2N+1\\
\sin \frac{a k \pi}{2(N+1)} &\mbox{ if } a=2,4,\dots, 2N  \end{cases} $$  
with associated eigenvalues 
$$\omega_{a,b} =  \frac{1}{4}\left( 2+ \cos \frac{a\pi}{2(N+1)}+\cos\frac{b\pi}{2(N+1)}\right)$$ 
when $a,b$ run over $\{1,2,\dots,2N+1\}$. 

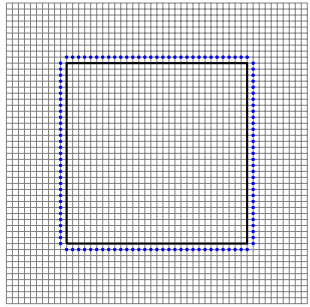
\begin{figure}[h] 
\begin{center}
\begin{tikzpicture}[scale=.08]
\draw [help lines] (0,0) grid (50,50);
\draw [thick] (10,10) -- (40,10) -- (40,40) -- (10,40) -- (10,10);

{\foreach \x in {0,1,2,3,4,5,6,7,8,9,10,11,12,13,14,15,16,17,18,19,20,21,22,23,24,25,26,27,28,29,30}
{\color{blue} \draw [fill] (10+\x,9) circle [radius=.2];
\draw [fill] (10+\x,41) circle [radius=.2];
\draw [fill] (9,10+\x) circle [radius=.2];
\draw [fill] (41,10+\x) circle [radius=.2];
}}

\end{tikzpicture}
\caption{The box $U=\{-N,\dots,N\}^2$ and its boundary. Each point on the boundary has exactly one neighbor in $U$.} \label{N2}
\end{center}
\end{figure}  
Applying\eqref{eq-g-eigen} and~\eqref{eq-norm-der}, we have
\begin{equation}\label{eq-g-box}
\frac{\partial_y g_U(x,y)}{\partial \vec{\nu}_U}  = \sum_{(a,b)\in \{1,\dots, 2N+1\}^2} (1-\omega_{a,b})^{-1} \phi_{a,b}(x)  \frac{\partial_y \phi_{a,b}(y)}{\partial \vec{\nu}_U}.
\end{equation}
To be more explicit, using the obvious symmetries of $U$, let's focus on the case when the boundary point $y=(y_1,y_2)$ is on the vertical, right side of $U$,
that is,  $y=(N+1,y_2)$ for $y_2\in \{-N,\dots,N\}$.  For such point, the neighbor of $y$ in $U$ is the point $\tilde{y}=(N,y_2)$ and so~\eqref{eq-g-box} becomes,
$$
\frac{\partial_y g_U(x,y)}{\partial \vec{\nu}_U}  = \frac{1}{8(N+1)} \sum_{(a,b)\in \{1,\dots, 2N+1\}^2} (1-\omega_{a,b})^{-1} \phi_{a,b}(x) \psi_b(y_2) \psi_a(N).$$
Writing this in a more explicit form, we have
\begin{eqnarray}
\lefteqn{P_U((x_1,x_2),(N+1,y_2))=}   &&\label{spP}\\
&&\frac{1}{4(N+1)^2}\sum_{(a,b)\in \{1,\dots, 2N+1\}^2} \frac{\psi_a(x_1)\psi_b(x_2) \psi_b(y_2) \psi_a(N)}{1-\frac{1}{2}\left(\cos \frac{a\pi}{2(N+1)}+\cos\frac{b\pi}{2(N+1)}\right)} . \nonumber \end{eqnarray}
\end{exa}   
There are several problems with formulas of the type (\ref{spP}).  The first is that  it is rare we can compute all eigenvalues and eigenvectors as in the above example. The second is that all the terms in the formula have roughly similar size and most are oscillating terms 
that change sign multiple times.  The terms that oscillate most are actually given somewhat higher weights in (\ref{spP}).  So, even in the case of the square domain treated above, it is not clear how much information one can extract from (\ref{spP}) except, perhaps, numerically.
   
\section{General results based on the Doob transform}   \label{sec-DT}
  
It is well-known that the Doob-transform technique is a useful tool to study problems involving Markov processes with killing.  We follow closely the notation used in our previous article \cite{DHSZ} which will be used extensively in what follows.

We work in the weighted graph setting introduced in Section \ref{sec-BC} and fix
a finite domain $U\subset \mathfrak X$. The operator associated to the sub-Markovian kernel~$K_U$,
$$ f\mapsto K_Uf=\sum_yK_U(\cdot,y)f(y),$$ 
acting on $L^2(U,\pi)$ admits a Perron-Frobenius eigenvalue $\beta_0$ and eigenfunction $\phi_0$ (because 
$K_U$ is selfadjoint on $L^2(U,\pi)$, right and left eigenvectors are the same). Here, we normalize $\phi_0$ by requiring that $\pi(\phi_0^2)=\pi|_U(\phi_0^2)=1$ (recall that the measure $\pi|_U$ is not normalized).
   
The Doob-transform technique  amounts to considering the Markov kernel
\begin{equation} \label{eq-doob-transform}
   K_{\phi_0}(x,y)=\beta_0^{-1}\phi_0(x)^{-1}K_U(x,y)\phi_0(y) 
\end{equation}
which is reversible with respect to the measure $\pi_{\phi_0}$, where we define $$\pi_{\phi_0}=\phi_0^2\pi|_U.$$
Just as $k^t_U(x,y)=K^t_U(x,y)/\pi(y)$, we set
$$k^t_{\phi_0}(x,y)= \frac{K^t_{\phi_0}(x,y)}{\pi_{\phi_0}(y)}.$$
This is the kernel of the operator $K^t_{\phi_0}$ with respect to its reversible measure $\pi_{\phi_0}.$
It is also clear that
$$k^t_U(x,y)=  \beta_0^t \phi_0(x)\phi_0(y)   k^t_{\phi_0}(x,y) . $$

Our basic assumptions imply that $K_U$ and $K_{\phi_0}$ are irreducible kernels, i.e., for any pair $x,y$, there is a $t=t(x,y)$ such that $K_U^t(x,y)>0$.  If we additionally assume that $K_{\phi_0}$ is aperiodic, this implies that the chain is ergodic. Hence, using these manipulations, we have reduced the study of $K_U^t$ to that of $K^t_{\phi_0}$,
the iterated kernel of an ergodic reversible finite Markov chain. In what follows, we do not assume aperiodicity, but it is often better to assume aperiodicity on the first reading in order to focus on the most interesting aspects of the computations and arguments involved.   This gives the following version of  Theorem \ref{th-Pab1}.
\begin{theo} \label{th-Pab2}
For $x\in U$ and $y\in \partial U$, the Poisson kernel $P_U(x,y)$ is given by  $P_U(x,y)=p_U(x,y)\pi(y)$ with
\begin{eqnarray*}p_U(x,y) &= & \phi_0(x) \sum_{t=0}^\infty \beta_0^t\frac{\partial_y \phi_0(y)k^t_{\phi_0}(x,y)}{\partial \vec{\nu}_U} \\
&=&  \phi_0(x) \sum_{t=0}^\infty \beta_0^t \sum_{z\in \nu(y)} \phi_0(z)k^t_{\phi_0}(x,z)\frac{\mu_{zy}}{\pi(y)}.\end{eqnarray*}
Similarly, $P_U(t,x,y)=p_U(t,x,y)\pi(y)$ with
\begin{eqnarray*}p_U(t,x,y)&= & \phi_0(x) \sum_{\ell=0}^{t-1} \beta_0^\ell\frac{\partial_y \phi_0(y)k^\ell_{\phi_0}(x,y)}{\partial \vec{\nu}_U} \\
&=&  \phi_0(x) \sum_{\ell=0}^{t-1} \beta_0^\ell \sum_{z\in \nu(y)} \phi_0(z)k^\ell_{\phi_0}(x,z)\frac{\mu_{zy}}{\pi(y)}.\end{eqnarray*}
\end{theo}
\begin{exa}[Example {\ref{exa-box1}}, continued] \label{exa-box2}
Let us spell out what Theorem \ref{th-Pab2} says in the case of the Euclidean box $U=\{-N,\dots,N\}^2\subset \mathbb Z^2$
depicted in Figure \ref{N2}.  The setting is as in Example \ref{exa-box1}. First, note that  the Perron-Frobenius eigenfunction
$\phi_0$ is given by
$$\phi_0(x)=\phi_0((x_1,x_2))=\frac{1}{(N+1)}\cos \frac {\pi x_1}{2(N+1)}\cos \frac{\pi x_2}{2(N+1)},$$
with associated eigenvalue
$$\beta_0=\frac{1}{2}\left(1+\cos \frac{\pi}{2(N+1)}\right)\sim 1-\frac{\pi^2}{16(N+1)^2},$$
where the asymptotic is  when $N$ tends to infinity. Using~\eqref{eq-doob-transform}, the associated Doob transform Markov chain has kernel
 $$K_{\phi_0}(x,y)=\left\{\begin{array}{cl} 0  & \mbox{ for } x,y\in U, |x_1-y_1|+|x_2-y_2|>1,\\
\frac{1}{2\beta_0} & \mbox{ for }  x,y\in U, x=y,\\
\frac{1}{8\beta_0} \frac{\phi_0(y)}{\phi_0(x) }& \mbox{ for } x,y\in U, |x_1-y_1|+|x_2-y_2|=1.
\end{array}\right.$$
By construction this kernel (which  resembles closely a Metropolis-Hastings kernel) is reversible with respect to the probability measure $\pi_{\phi_0}$.  It is also irreducible and aperiodic and thus, for any $x,y\in U$, $$K_{\phi_0}^t(x,y)\rightarrow \pi_{\phi_0}(y)$$ as $t$ tends to infinity.  Equivalently, $k^t_{\phi_0}(x,y) \rightarrow 1$ as $t$ tends to infinity.   Recall that each boundary point $y\in \partial U$ has exactly one neighbor $y^*$ in $U$.  Using this information, the Poisson kernel formula
provided by Theorem \ref{th-Pab2} reads
$$p_U(x,y)
=  \frac{1}{8}\phi_0(x) \phi_0(y^*) \sum_{t=0}^\infty \beta_0^t k^t_{\phi_0}(x,y^*),
 \;\;x\in U,\; y\in \partial U.$$
This makes it clear that a two-sided bound for $p_U(x,y)$, valid for all $x\in U$ and $y\in \partial U$, would follow from a two-sided bound on $k^t_{\phi_0}(x,y^*)$ that holds uniformly
in $t,x,$ and $y^*$.  Such a bound is provided in the next two sections.
  \end{exa}
     
 \section{Harnack Markov chains and Harnack weighted graphs}\label{sec-Harnack}
 
 In this section, we discuss the highly non-trivial notion of a {\em Harnack Markov chain} or, equivalently,  of a {\em Harnack weighted graph}.   
 Consider a weighted graph  $(\mathfrak X,\mathfrak E, \pi,\mu)$ satisfying (\ref{cond-pimu}) and its associated Markov kernel $K$ defined at (\ref{def-K}).
For $x,y\in \mathfrak X$, let $d(x,y)$ be the minimal number of edges in $\mathfrak E$ one must cross to join $x$ to $y$ by a discrete path. Let
$$B(x,r)=\{y\in \mathfrak X: d(x,y)\le r).$$
be the ball of radius $r$ around $x \in \mathfrak{X}$. Note that $B(x,r)\cup \partial B(x,r)=B(x, r+1)$. 
 
 Fix a parameter $\theta\ge 2$ (it turns out that the assumption that $\theta\ge 2$ is not restrictive for what follows). The key point in the following definition is 
 that the constant $C_H$  is required to be independent of scale and location (i.e., $R\ge 1$, $t_0 \in \mathbb N$ and $x_0\in \mathfrak X$) and also independent of the non-negative function $u$, the solution of (\ref{def-he1}). 
 \begin{defin} \label{def-HC}
 We say that $(K,\pi)$ is a $\theta$-Harnack Markov chain (equivalently, that $(\mathfrak X,\mathfrak E, \pi,\mu)$ is a $\theta$-Harnack weighted graph),
 if there exists a constant $C_H$ such that  for any $R>0$, $t_0 \in \mathbb N$, and $x_0 \in \mathfrak X$, and non-negative function $u:\mathbb{N} \times \mathfrak{X} \rightarrow \mathbb{R}_{\geq 0}$ defined on a time-space cylinder
 $$Q(R,t_0,x_0)=\left[t_0, t_0+4\lceil R^\theta\rceil+1\right]\times B(x_0, 2R+1) $$
such that
\begin{equation} u(t+1,x)= \sum_yu(t,y)K(x,y)  \label{def-he1}
\end{equation}
in $$Q'(R,t_0,x_0)= \left[t_0, t_0+4\lceil R^\theta\rceil \right]\times B(x_0, 2R),$$ 
it holds that, for all $(t,x)\in Q_-(R,t_0,x_0)= [t_0+\lceil R^\theta\rceil, t_0+2\lceil R^\theta\rceil]\times B(x_0, R),$
$$u(t,x)\le C_H \min_{(k,y)\in Q_+(R,t_0,x_0)}\{u(k,y)+u(k+1,y)\}$$
where 
$$Q_+(R,t_0,x_0)=  \left[t_0+3\lceil R^\theta\rceil, t_0+4\lceil R^\theta\rceil\right]\times B(x_0, R) .$$
 \end{defin}
 Equation (\ref{def-he1}) can also be written using the graph Laplacian $\Delta=I-K$ (i.e., $\Delta u(t,x)= u(t,x)-\sum_y K(x,y)u(t,y)$)
  and the time difference operator $\partial_t u(t,x)=u(t+1,x)-u(t,x)$
 in the form
 \begin{equation} \label{def-he2}
\partial_t u +\Delta u=0.
 \end{equation} 
This is the discrete-time heat equation on $(\mathfrak X,\mathfrak E,\mu,\pi)$ and the property required in Definition \ref{def-HC} is the validity, at all scales and locations, of the discrete time $\theta$-parabolic Harnack inequality.  

\begin{exa} The square lattice $\mathbb Z^n$, equipped with the vertex weight $\pi\equiv 1$ and the edge weight   $\mu\equiv 1/2^n$, on $\mathfrak E$ is a
$2$-Harnack weighted graph. See~\cite{Barlow,Delm-PH,GrigoryanAMS}.
\end{exa} 
\begin{exa} The Sierpinski gasket graph is a $\theta$-Harnack weighted graph with $\theta=\log5/\log2$. See, e.g., \cite[Section 2.9 and Corollary 6.11]{Barlow} and \cite{BB}. 
 \end{exa}
 These two examples illustrate the fact that $\theta=2$  corresponds to the more classical situation of $\mathbb Z^n$ when the random walk has a diffusive
 behavior in the sense that it travels approximately a distance  $\sqrt{t}$ in time $t$ whereas the case $\theta>2$ corresponds to sub-diffusive behaviors when the random walk travel approximately a distance  $t^{1/\theta} < \sqrt{t}$ in time $t$. This second type of behavior is typical of fractal type spaces.  The following theorem make these statement more precise.
  
  \begin{theo}[See {\cite[Theorem 3.1]{GT2} and also \cite[Theorem 1.2]{BB}}]   \label{theo-PHItheta}
  Assume that the weighted graph $(\mathfrak X,\mathfrak E, \pi,\mu)$ satisfies the ellipticity condition 
  \begin{equation}
  \label{elliptic} 
  \forall \{x,y\}\in \mathfrak E,\;\;\pi(x)\le P_e \mu_{xy}\end{equation}
   for some fixed constant $P_e$.  Under this assumption,  $(\mathfrak X,\mathfrak E, \pi,\mu)$ is a $\theta$-Harnack 
 graph if and only if the iterated transition kernel $k^t(x,y)=K^t(x,y)/\pi(y)$ of the chain $(K,\pi)$ satisfies
 \begin{equation} \label{ktup}
  k^t(x,y)\le \frac{C_1}{\pi(B(x,t^{1/\theta}))} \exp\left(- c_1 \left(\frac{d(x,y)^\theta}{t}\right)^{1/(\theta-1)}\right) 
   \end{equation}
  when $ d(x,y)\le n$, and
    \begin{equation} \label{ktlow}  
    k^{t+1}(x,y)+k^t(x,y)\ge  \frac{c_2}{\pi(B(x,t^{1/\theta}))} \exp\left(- C_2 \left(\frac{d(x,y)^\theta}{t}\right)^{1/(\theta-1)}\right), \end{equation}
    where $c_1,c_2,C_1,C_2>0$.
     \end{theo}
Theorem \ref{theo-PHItheta} established the equivalence of two properties, each of which seems (and is) very hard to verify.  The following theorem offers a third equivalent condition which, at least in the case $\theta=2$, can sometimes be checked using elementary arguments.
 
   \begin{theo}[See {\cite[Theorem 1.5]{BB}}]  \label{th-D-BB}
    Assume that the weighted graph $(\mathfrak X,\mathfrak E, \pi,\mu)$ satisfies the ellipticity condition {\em (\ref{elliptic})} for some fixed constant $P_e$.  Under this assumption,  $(\mathfrak X,\mathfrak E, \pi,\mu)$ is a $\theta$-Harnack 
 graph if and only if  the following three conditions are satisfied:
 \begin{enumerate} 
 \item There is a constant $C_D$ such that, for all $x\in \mathfrak X$ and all $r>0$, 
 $$\pi(B(x,2r)) \le C_D \pi(B(x,r)).$$
 In words, the volume doubling condition is satisfied. 
   \item There is a constant $C_P$ such that, for all $x\in \mathfrak X$ and all $r>0$,  the Poincar\'e inequality with constant $C_P r^\theta$ holds on the ball $B(x,r)$, i.e.,
   $$\forall\, f,\;\;\sum_{z\in B}|f(z)-f_B|^2\pi(z)\le C_Pr^\theta \sum_{\xi,\zeta\in B, (\xi\zeta)\in\mathfrak E}|f(\xi)-f(\zeta|^2\mu_{\xi\zeta},$$
 where $f_B=\pi(B)^{-1}\sum_Bf\pi$.
   \item  The cut-off function existence property $\mbox{\em CS}(\theta)$ is satisfied. (See Definition~\ref{defn-cs} below.)
 \end{enumerate}    
When $\theta=2$, the cut-off function existence property $\mbox{CS}(\theta)$ is always satisfied. \end{theo}

\begin{defin}[{\cite[Definition 1.4]{BB}}]\label{defn-cs} Fix $\theta\in [2,\infty)$. The weighted graph 
$(\mathfrak X,\mathfrak E, \pi,\mu)$ satisfies
the cut-off function existence property $\mbox{\em CS}(\theta)$ if there are constants $C_1,C_2,C_3$ and $\epsilon>0$ such that, for any $x\in \mathfrak X$ and $r>0$, there exists a function $\phi=\phi_{x,r}$  satisfying the following four properties:
\begin{itemize}
\item[(a)] $\phi\ge 1$ on $B(x,r/2)$
\item[(b)] $\phi\equiv 0$ on $\mathfrak X\setminus B(x,r)$
\item[(c)] For all $y,z\in\mathfrak X$, $|\phi(z)-\phi(y)|\le C_1(d(z,y)/r)^\epsilon$
\item[(d)] For any $s\in (0,r]$ and any function $f$ on $B(x,2r)$,
\begin{eqnarray*}
\lefteqn{\sum_{z\in B(x,s)}|f|^2\sum_{y: \{z,y\}\in \mathfrak E}|\phi(z)-\phi(y)|^2\mu_{zy}}&&\\
&\le &C_2(s/r)^{2\epsilon}\left\{\sum_{{z,y\in B(x,2s)}\atop {\{z,y\}\in \mathfrak E}}|f(z)-f(y)|^2\mu_{zy} +s^{-\theta}\sum_{ B(x,2s)}|f|^2\pi \right\}.\end{eqnarray*}
\end{itemize}
\end{defin} 
 \begin{rem} \label{rem-CS}
 Given the rather unwieldy nature of this definition, some comments are in order. When $\theta=2$, the function $\phi(z)=\min\{1, 2(1-d(x,z)/r)_+\}$ provides the desired cut-off function. In that case, the inequality in (d) contains no particularly interesting information (it does say, for $s$ near $1/2$, that $\sum_y\mu_{xy}\le 4\pi(x)$,
 which is weaker than our basic assumption $\sum_y\mu_{xy}\le \pi(x)$).
 
 For $\theta>2$, the inequality in (d) becomes the carrier of some  (somewhat mysterious) useful information. One of its simplest consequences is a lower estimate for the Perron-Frobenius eigenvalue $\beta_0=\beta_{U,0}$ when $U=B(x,r)$.  Namely, the cut-off function $\phi$ for the ball $B(x,r)$ must satisfy
 $$\pi (|\phi|^2)\ge  \pi(B(x,r/2))$$
by (a) in Definition~\ref{defn-cs} and
 $$\sum_{{z,y\in B(x,r)}\atop {\{z,y\}\in \mathfrak E}}|\phi(z)-\phi(y)|^2\mu_{zy}
\le C_2r^{-\theta} \pi(B(z,2r))$$
by (d) in Defintion~\ref{defn-cs}, taking $f \equiv 1$ and $s=r$. Together with the doubling property, this implies that the Perron-Frobenius eigenvalue of the ball $B(x,r)$ satisfies
\begin{equation}\label{beta0theta}
1-\beta_{B(x,r),0} \le \frac{C_2C_D^2}{r^\theta}.\end{equation}
 \end{rem}
 
 The aim of the next theorem is to illustrate in the simplest possible way the use of the notion of a Harnack Markov chain  in obtaining two-sided estimates 
 on $p_U(x,y)$.  We introduce the following definition and notation.
 
 \begin{defin} For any finite domain $U$ in $\mathfrak X$, let
 $(U,\mathfrak E_U)$ be the associated subgraph with edge set $\mathfrak E_U=\{(x,y)\in \mathfrak E: x,y\in U\}$.   
 Let $d_U$ be the associated graph distance and $B_U$ the corresponding graph balls. If $\beta_0,\phi_0$ are the Perron-Frobenius eigenvalue and eigenfunction for $U$ on $(\mathfrak X,\mathfrak E,\mu,\pi)$, the Markov chain $(K_{\phi_0},\pi_{\phi_0})$ is the chain associated with the weighted graph
 $$( U,\mathfrak E_U, \mu^{\phi_0}, \pi_{\phi_0}) \mbox{ where } \mu^{\phi_0}_{xy}=\beta_0^{-1}\phi_0(x)\phi_0(y)\mu_{xy}.$$
 \end{defin}
 
\begin{rem}
We use $A(t,x,y) \approx B(t,x,y)$ when there exists $c,C>0$ such that $$c \leq \frac{A(t,x,y)}{B(t,x,y)} \leq C,$$ where $c,C$ depend only on the key parameters (e.g., dimension, and the constants from volume doubling, the Harnack condition, and the Poincar\'e inequality) and not on the specific time $t$, positions $x,y$, or any size parameters (e.g., $r$ where $x, y \in B(z,r)$). When there is a subscript on $\approx$ (such as $\approx_{\epsilon}$ or $\approx_n$) the constants $c,C$ additionally depend on the parameter in the subscript.
\end{rem} 
 
 \begin{theo} \label{theo-PUo}
 Let $U$ be a finite domain in $(\mathfrak X,\mathfrak E,\pi,\mu)$ with Perron-Frobenius eigenvalue and eigenfunction $\beta_0,\phi_0$.
 Let $T_U$ be such that $\beta_0=1-1/T_U$. Assume that 
\begin{enumerate}
\item  There exists $C\geq 0$, $R \in \mathbb Z$ and a point $o\in U$ such that   $$B(o,R/2)\subset U \mbox{ and  }  \;U\subset B_U(o,CR);$$
\item The weighted graph  $(\mathfrak X,\mathfrak E,\mu,\pi)$ is a $\theta$-Harnack weighted graph which satisfies 
the ellipticity condition $\pi(x)\le P_e \mu_{xy}$ for some fixed constant $P_e$. 
 \item The Markov chain $(K_{\phi_0},\pi_{\phi_0})$ is a $\theta$-Harnack chain on  $(U,\mathfrak E_U)$.\end{enumerate}
Under these assumptions, for any point $y$ on the boundary $\partial U,$
$$P_U(o,y) \approx  T_U \phi_0(o)\sum_{z\in \nu(y)} \phi_0(z) \mu_{zy} \approx \frac{T_U}{\sqrt{\pi(U)}} \sum_{z\in \nu(y)} \phi_0(z) \mu_{zy}.$$
 \end{theo}
 
 \begin{proof}  First, we start with remarks regarding $\phi_0(o)$. By assumption, the measure $\pi_{\pi_0}$ is doubling and $\pi(\phi_0^2)=1$. It follows that, for any fixed $\epsilon\in (0,1/2)$,  
 $$\sum_{B(o, \epsilon R)}\phi_0^2\pi \approx _\epsilon 1.$$ Because $(\mathfrak X,\mathfrak E,\mu,\pi)$ is a $\theta$-Harnack weighted graph,  $\phi_0(o)\approx _\epsilon \phi_0(z)$ for any $z\in B(o,\epsilon R)$. Using this and the doubling property of $\pi_{\phi_0}$,
 \begin{equation}
     \label{eq:phi0approx}
 \phi_0(o)^2  \approx _\epsilon \pi(B(o,\epsilon R)) ^{-1} \sum_{B(o, \epsilon R)}\phi_0^2\pi \approx_\epsilon \pi(U)^{-1}.
 \end{equation}
 Using~\eqref{eq:phi0approx} and the doubling property of $\pi$,
 $$\pi(U)^{1/2} \approx \phi_0(o) \pi(U) \approx_\epsilon \sum_{B(o, \epsilon R)}\phi_0 \pi \le \pi(\phi_0).$$  Also, $\pi(\phi_0)^2\le \pi(U)\pi(\phi_0^2)=\pi(U)$.
 It follows that
  $$  \pi (U) \phi_0(o)^2 \approx 1 \mbox{ and }\pi(\phi_0)\approx \pi(U)^{1/2}. $$
  We need to estimate (see Theorem \ref{th-Pab2})
  $$ P_U(o,y)=p_U(o,y)\pi(y)=\phi_0(o)\sum_{z\in \nu(y)}\mu_{zy}\phi_0(z) \left(\sum_{t=d_U(o,z)}^\infty \beta_0^t k^t_{\phi_0}(o,z) \right).$$
  
 Because of the second hypothesis, $\beta_0=1-1/T_U\ge 1- CR^{-\theta}$  and $R^\theta\le CT_U$  (see Remark \ref{rem-CS}). It follows that $(K_{\phi_0},\pi_{\phi_0})$ also satisfies the ellipticity condition and thus $(K_{\phi_0},\pi_{\phi_0})$ is a Harnack Markov chain satisfying the ellipticity condition and we can use the heat kernel estimates of 
 Theorem \ref{theo-PHItheta}.  In the bounds in (\ref{ktup})-(\ref{ktlow}), the distance $d$ is now $d_U$.  We observe that, for $z\in \nu(y)$, $R/2\le d_U(o,z)\le CR$
 and (using the doubling property of $\pi_{\phi_0}$ and the normalization $\pi(\phi_0^2)=1$),
 $$\sum_{t=d_U(o,z)}^{R^\theta} \frac{1}{\pi_{\phi_0}(B_U(o,t^{1/\theta})) } e^{- c (R^\theta/t)^{1/(\theta-1)} }\approx R^\theta.$$
It follows that
$$\sum_{t=d_U(o,z)}^\infty \beta_0^t k_{\phi_0}(o,z) \approx \sum_{t=d_U(o,z)}^{R^\theta} k^t_{\phi_0}(o,z)  +\sum_{t>R^\theta} \beta_0^t \approx T_U $$
because, for $t\ge R^\theta$, we have  $k_{\phi_0}^t(o,z) +k_{\phi_0}^{t+1}(o,z)\approx 1$, and, for $t\le R^\theta$, $\beta_0^t\approx 1$. 
Also
$\beta_0\in (0,1)$ and $\sum_{t>R^\theta}\beta_0^t\approx \frac{1}{1-\beta_0} \beta_0^{R^\theta}\approx T_U $. \end{proof}
 \begin{rem} By definition, the quantity  $P_U(o,\cdot)$ defines a probability measure on $\partial U$. This means that it must be the case that, under the hypotheses of Theorem \ref{theo-PUo},  
 \begin{equation} \label{eq-prob}
 T_U \phi(o)\sum_{y\in \partial U}\sum_{z\in\nu(y)}\phi_0(z)\mu_{zy}\approx 1.\end{equation}
 To verify  that this is indeed the case, observe  (extending $\phi_0$ by $0$ outside of $U$ and using the scalar product on $L^2(\mathfrak X,\pi)$)
 \begin{eqnarray*}
 \langle \mathbf 1_U,(I-K)\phi_0\rangle _\pi &= & \sum_{\{x,y\}\in \mathfrak E} (\mathbf 1_U(x)-\mathbf 1_U(y))(\phi_0(x)-\phi_0(y))\mu_{xy} \\
& =& \sum_{y\in \partial U}\sum_{z\in \nu(y)} \phi_0(z)\mu_{zy} .\end{eqnarray*}
 It follows that
 \begin{equation}
     \label{eqn-tu-est}
 \sum_{y\in \partial U}\sum_{z\in \nu(y)} \phi_0(z)\mu_{zy}=\sum_U (I-K_U)\phi_0 \pi =(1-\beta_0)\sum_U \phi_0 \pi=T_U^{-1}\pi(\phi_0).
 \end{equation}
 The estimate~\eqref{eq-prob} now  follows from~\eqref{eqn-tu-est} and $$\phi_0(o) \approx \pi(U)^{-1/2},\;\; \pi(\phi_0)\approx  \pi(U)^{1/2}.$$
 \end{rem}

 \begin{exa}[Example {\ref{exa-box1},\ref{exa-box2}}, continued]   \label{exa-box3}Theorem \ref{theo-PUo}  can be applied to the Euclidean box $U=\{-N,\dots,N\}^2\subset \mathbb Z^2$
depicted in Figure \ref{N2}.  The explicit Perron-Frobenius  eigenvalue and eigenfunction 
$\beta_0,\phi_0$ are  given above in Examples \ref{exa-box1} and~\ref{exa-box2}.  The square grid $\mathbb Z^2$ is one of the basic examples of a $2$-Harnack graph. It also turns out that $(K_{\phi_0},\pi_{\phi_0})$  is a $2$-Harnack Markov chain on $U$, which can proved using Theorem \ref{th-D-BB} with $\theta=2$. (This is a theorem due to Thierry Delmotte  \cite{Delm-PH} in the case $\theta=2$.)  See, e.g., \cite{DHSZ}.   This gives, for $n\in \{-N,\dots,N\}$,
\begin{eqnarray*}
P_U((0,0),(N+1,n))&\approx & \frac{(N+1)^2}{(N+1)^2} \cos\frac{\pi N}{2(N+1)}\cos \frac{\pi n}{2(N+1)}\\
& =& \sin \frac{\pi}{2(N+1)} \cos \frac{\pi n}{2(N+1)} \\
&\approx& \frac{1}{(N+1)} \cos \frac{\pi n}{2(N+1)}.\end{eqnarray*}
\end{exa}

\begin{exa} \label{exa-Boxd}
We spell out how the preceding example generalizes in dimension $n$ when $U=\{-N,\dots,N\}^n$. Here the graph $\mathbb Z^n$ is equipped with the edge weight $\mu_{xy}=1/4n$ if $\sum_1^n|x_i-y_i|=1$ and $0$ otherwise and the vertex weight $\pi\equiv 1$.   As in dimension $2$, one can compute exactly 
$$\beta_0= \frac{1}{2}\left(1+\cos\frac{\pi}{2(N+1)}\right)$$
and
$$\phi_0(x_1,\dots,x_n)= \frac{1}{(N+1)^{n/2}} \cos\frac{\pi x_1}{2(N+1)}\cdots \cos\frac{\pi x_n}{2(N+1)}.$$
Observe that a point $y=(y_1,\dots,y_n)$ is on the boundary $\partial U$ of $U$ if and only if there is a $j\subset \{1,\dots,n\}$  such that  $y_j=N+1$
and all other coordinates of $y$ are in $\{-N,\dots,N\}$. For such a $y$,
\begin{eqnarray*}
P_U(0,y)& \approx_n & \frac{(N+1)^2}{(N+1)^n} \sin \frac{\pi}{2(N+1)}\prod_{i\neq j} \cos \frac{\pi y_i}{2(N+1)} \\
&\approx_n& \frac{1}{(N+1)^{n-1}}\prod_{i\neq j} \cos \frac{\pi y_i}{2(N+1)}.
\end{eqnarray*}
\end{exa}

 \begin{figure}[h] 
\begin{center}
\begin{tikzpicture}[scale=.2]
\draw [help lines] (0,0) grid (20,20);

\foreach \x in {1,2,3,4,5,6,7,8,9,10,11,12,13,14,15,16,17,18,19,20}
{\draw  (\x,0) -- (0,\x);
\draw (\x,0) -- (\x,20-\x);
\draw (0,\x)--(20-\x,\x);
}
\draw (0,0)--(0,20);
\draw (0,0)-- (20,0);

\foreach \x in {0,1,2,3,4,5,6,7,8,9,10,11,12,13,14,15,16,17,18,19,20}
{\draw [fill, blue] (\x,0) circle [radius=.2];\draw  [fill, blue](0,\x)  circle [radius=.2];}
\foreach \x in {1,2,3,4,5,6,7,8,9,10,11,12,13,14,15,16,17,18,19}{\draw [fill, blue] (20-\x,\x) circle [radius=.2];}

\draw[fill,white]  (7,7) circle  [radius=2]; 
\node at (7,7)  {$\phi_0$};

\draw[fill,white]  (13,13) circle  [radius=2]; 

\node at (13,13)  {$-\phi_0$};

\end{tikzpicture}
\caption{The gambler's ruin problem with $3$ players: extending  the Perron-Frobenius eigenfunction $\phi_0$ into a global $N\mathbb Z^2$ periodic eigenfunction. The function $\phi_0$ vanishes at the blue dots.} \label{GR2phi}\end{center}
\end{figure}
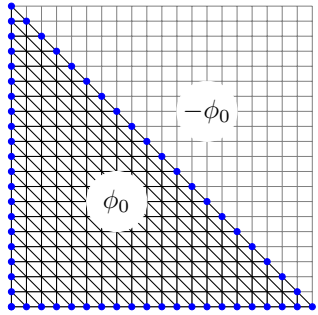  

We end this section with the treatment of the ($2$-dimensional) $3$-player gambler's ruin problem depicted in Figure \ref{GR2}.
\begin{exa}[The $3$-player gambler's ruin problem]  \label{exa-GR}
The notation is described in the introduction. Theorem \ref{theo-PUo} applies to the $3$-player gambler's ruin problem (see Section~\ref{sec-hitting}).  In this case, as in the other examples discussed above, it is possible to compute the Perron-Frobenius eigenfunction exactly. This is related to the fact that the eigenfunctions of (Euclidean) equilateral triangles can be computed in closed trigonometric form, a fact first observe by Lam\'e. See the related history in \cite{McC3} and the treatment in \cite{McC0,McC1,McC2}. We explain the computation in detail in the square lattice coordinate system for the convenience of the reader.

First, we compute $\phi_0$ and $\beta_0$ (this is possible in closed form only in dimension $2$). Note that $\phi_0$, being the unique Perron-Frobenius eigenfunction (up to a multiplicative constant), must be symmetric with respect to swapping the two coordinates. We extend $\phi_0$ into a function defined in the entire square $\{0,\dots,N\}^2$ so that the symmetry with respect to  $x_1+x_2=N$  changes the extended $\phi_0$ into $-\phi_0$ (and we still call this extension $\phi_0$). We then extend this function to the entire grid $\mathbb Z^2$ by using translations by $N\mathbb Z^2$.  We now have a function defined on all of $\mathbb Z^2$ and, by construction,  this function is a  $N\mathbb Z^2$ periodic solution of  $K\phi_0=\beta_0\phi_0$ where $K$ given 
for all pairs $(x_1,x_2),(y_1,y_2)\in \mathbb Z^2$ is given by
$$K((x_1,x_2),(y_1,y_2))= \left\{\begin{array}{cl} 1/6  & \mbox{ if }  |x_1-y_1|+|x_2-y_2|=1,\\
 1/6 & \mbox{  if  }  x_1-y_1=y_2-x_2=\pm 1,\\
 0& \mbox{ otherwise.}\end{array}\right.$$ 
 
Global periodic solutions of the equation  $K\phi=\beta\phi$ must be linear compositions of functions of the type   $e^{i y\cdot x}$ with
$$\beta= \frac{1}{3}\left(\cos a+\cos b+2+ \cos (a-b)\right), (a,b)\in \frac{2\pi}{N} \mathbb Z^2 $$
Constant functions correspond to   $a=b=0$.   The second smallest eigenvalue for this  problem is
$$\beta= \frac{1}{3} \left(1+2\cos \frac{2\pi}{N}\right)$$
with a 6 dimensional  real eigenspace spanned by
$$\sin \frac{2\pi x_1}{N}, \sin \frac{2\pi x_2}{N}, \sin \frac{2\pi (x_1+x_2)}{N} $$
and their cosine counterparts (which we will not use).   In this eigenspace, consider the function 
\begin{eqnarray}\phi((x_1,x_2))&=& \sin \frac{2\pi x_1}{N}+ \sin \frac{2\pi x_2}{N} -\sin \frac{2\pi (x_1+x_2)}{N} \\
&=&   \sin \frac{2\pi x_1}{N} \left(1-\cos \frac{2\pi x_2}{N}\right)+ \sin \frac{2\pi x_2}{N} \left(1-\cos \frac{2\pi x_1}{N}\right) \nonumber \\
&=&\sin \frac{2\pi x_1}{N}+ \sin \frac{2\pi x_2}{N} +\sin \frac{2\pi (N-(x_1+x_2))}{N}  .
\nonumber \nonumber\end{eqnarray}
This function vanishes when $x_1=0$, when $x_2=0$ and also when $x_1+x_2=N$. Furthermore, by careful inspection, $\phi\ge 0$ in the triangle 
$$U\cup \partial U=\{(x_1,x_2): 0\le x_1,\; 0\leq x_2,\;x_1+x_2\le N\}.$$  It follows that
it must be the case  that
$$\beta_0=  \frac{1}{3} \left(1+2\cos \frac{2\pi}{N}\right) $$
and  
$$\phi_0((x_1,x_2))=\frac{2}{\sqrt{3}N}\left(\sin \frac{2\pi x_1}{N}+ \sin \frac{2\pi x_2}{N} -\sin \frac{2\pi (x_1+x_2)}{N}\right).$$

The following uniform two-sided  estimate captures some of the essential information regarding the behavior of $\phi_0$, namely,
\begin{equation}\label{GRphi}
\phi_0((x_1,x_2))\approx \frac{1}{N^7}x_1x_2(x_1+x_2)(N-x_1)(N-x_2)(N-(x_1+x_2)).
\end{equation}
This captures all the symmetries of the problem. The value of $\phi_0$ at the central point $([N/4],[N/4])$ is roughly $\frac{1}{N}$ as expected (i.e., $1/\sqrt{\pi(U)}$). If one approaches any of the three corners along its median,
$\phi_0$ vanishes as the cube of the distance to the corner. For the vertical part of the boundary, $\{(0,y): 1\le y<N\}$, Theorem \ref{theo-PUo}
gives, 
$$P_U(([N/4],[N/4]),(0,y)) 
\approx  \frac{N^2}{N} \frac{y^2N(N-y)^2}{N^7}\approx \frac{y^2(N-y)^2}{N^5}.$$
Of course a similar formula holds for the other two sides of the triangle.  Along the diagonal side $\{(x,N-x):1 \leq x < N\}$, the formula reads
$$P_U(([N/4],[N/4]),(x,N-x))\approx   \frac{x^2(N-x)^2}{N^5}.$$ In Section~\ref{sec-ex} we complete the description of harmonic measure, giving approximations valid for all starting positions. \end{exa}

\section{Inner-uniform domains and global two-sided estimates}\label{sec-inner-uniform}

\subsection{Inner-uniform domains}\label{subsec-inner-uniform}
 
We now describe a large class of domains for which the hypotheses of Theorem \ref{theo-PUo} can be verified thanks to the results obtained by the authors in \cite{DHSZ}.
 For an inner-uniform domain (described below), we amplify Theorem \ref{theo-PUo} by giving two sided estimates 
 of $P_U(x,y)$ which are uniform in $x\in U$ and $y\in \partial U$.
 
The following definition is well-known in the context of Riemannian and conformal geometry. See~\cite{DHSZ} for a more complete discussion and pointers to the literature. All the domains discussed in Examples \ref{exa-box3}--\ref{exa-GR} in the previous section are inner-uniform (in a rather trivial way).

\begin{defin} A domain $U \subseteq \mathfrak X$ is an inner $(\alpha,A)$-uniform domain (with respect to the graph structure $(\mathfrak X,\mathfrak E)$)
	if  for any two points $x,y\in U$ there exists a path $\gamma_{xy}=(x_0=x,x_1,\dots,x_k=y)$ joining $x$ to $y$ in  $(U, \mathfrak E_U)$ with 
	the properties that:
	\begin{enumerate}
		\item  $k  \le A d_U(x,y)$;
		\item  For any $j \in \{0,\ldots, k\}$, $d (x_j,\mathfrak X \setminus U) \ge  \alpha (1+\min \{ j,k-j\})$. 
	\end{enumerate}
\end{defin}

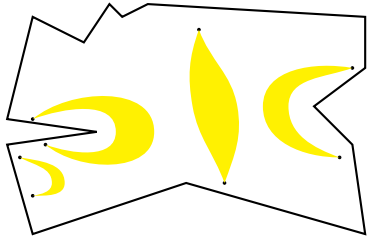
\begin{figure}[h] 
\begin{center}
\begin{tikzpicture}[scale=.17]
\draw[thick]  (8,10) -- (1,11) --(3,19) --( 7,17)--(9,20) --(10,19)--(12,20)--(29,19)-- (29,15)-- (25,12)-- (28,9)--( 29,2) -- 
(15,6)-- (3,2) -- (1,9) -- (8,10);

\draw[fill] (2,8) circle [radius=.1];
\draw[fill] (3,5) circle [radius=.1];
\path[fill=yellow]  (2,8) to [out=-10, in= 90]  (5.5,6)  to [out=-90, in=0]  (3,5)  to [out =10, in=-90] (4.5,6)  to [out=90, in=-30] (2,8) ;

\draw[fill] (4,9) circle [radius=.1];
\draw[fill] (3,11) circle [radius=.1];
\path[fill=yellow]  (4,9) to [out=-40, in= -90]  (12.5,10)  to [out=90, in=40]  (3,11)  to [out =20, in=90] (9.5,10)  to [out=-90, in=-20] (4,9) ;

\draw[fill] (28,15) circle [radius=.1];
\draw[fill] (27,8) circle [radius=.1];
\path[fill=yellow]  (28,15) to [out=170, in= 90]  (21,12)  to [out=-90, in=-180]  (27,8)  to [out =160, in=-90] (23,12)  to [out=90, in=-160] (28,15) ;

\draw[fill] (16,18) circle [radius=.1];
\draw[fill] (18,6) circle [radius=.1];
\path[fill=yellow]  (16,18) to [out=-70, in= 100]  (19,12)  to [out=-80, in=70]  (18,6)  to [out =110, in=-80] (15.5,12)  to [out=100, in=-110] (16,18) ;

\end{tikzpicture}
\caption{An illustration of the inner-uniform condition. Note the banana-shaped region between any two points in $U$.}\label{Banana}
\end{center}
\end{figure}  

Intuitively, $U$ is an inner-uniform domain if, given any two points ${x,y \in U}$, one can form a banana-shaped region between $x$ and $y$ which is entirely contained in $U$. (See Figure~\ref{Banana} for an illustration.) The following is a simple geometric consequence of the definition of inner-uniform domains.
\begin{lem}\label{lem-inner-point} Let $U$ be a finite inner $(\alpha,A)$-uniform domain.  Set 
$$R=\max\{x\in U: d(x,\mathfrak X\setminus U)\}.$$
There are constants $a_1,A_1$ depending only on $\alpha,A$ such that, for  any point $o$ such that $d(o,\mathfrak X\setminus U)=R/2$,  we have
 $$B(o, a_1R)\subset U\subset B_U(o,A_1R).$$
Furthermore, for any point $x\in U$ and any $r>0$, there is a point $x_r\in U$  such that
$$d_U(x,x_r)\le A_1\min\{ r, R\} \mbox{ and }  d(x_r,\mathfrak X\setminus U\})\ge a_1 
\min\{r,R\}.$$
\end{lem}
\begin{rem} In what follows, for each $x\in U$ and $r>0$, we fix a point $x_r$ with the properties stated above. The exact choice of these $x_r$  among all points with the desired properties is unimportant.
Typically, for $r\ge R$, we pick $x_r=o$. See~\cite{DHSZ} for a proof of the existence of such a point.
\end{rem}

\begin{theo}[\cite{DHSZ}] \label{th-H}
Fix  $\alpha\in (0,1]$ and $A\ge 1$.
Assume that $(\mathfrak X,\mathfrak E,\pi,\mu)$ is a $2$-Harnack graph satisfying the ellipticity condition {\em (\ref{elliptic})} and that $U$ is a finite inner $(\alpha,A)$-uniform domain with Perron-Frobenius  eigenvalue and  eigenfunction
$\beta_0,\phi_0$.  Then the chain $(K_{\phi_0},\pi_{\phi_0})$ on $(U,\mathfrak E_U)$
is a $2$-Harnack chain  with Harnack constant depending only on $C_H$, the Harnack constant of $(\mathfrak X,\mathfrak E,\pi,\mu)$, the ellipticity constant $P_e$ and the inner-uniformity constants $\alpha, A$. 
\end{theo}
\begin{proof}[Outline of the proof] The proof consists in showing that the weighted graph
on $(U,\mathfrak E_U)$ associated with $(K_{\phi_0},\pi_{\phi_0})$ satisfies the doubling condition and the Poincar\'e inequality on balls with constant $Cr^2$, where $C$ depends only on $C_H,P_e,\alpha$ and $A$. Once this is done, the result follows from Theorem \ref{th-D-BB} (in the case $\theta=2$ used here, the result is due to Delmotte). One of the keys to proving the desired  doubling and Poincar\'e inequality on balls is the following Carleson type estimate for $\phi_0$. We state this result because of its importance and also because it allows us to compute the volume for $\pi_{\phi_0}$ in a more explicit way.
\end{proof}
 \begin{theo}[\cite{DHSZ}] \label{th-C}
 Assume that $(\mathfrak X,\mathfrak E,\pi,\mu)$ is a $2$-Harnack graph satisfying the ellipticity condition {\em (\ref{elliptic})} and that $U$ is a finite inner $(\alpha,A)$-uniform domain with Perron-Frobenius  eigenfunction
$\phi_0$.    Then there is a constant  $C_U$ depending only on $C_H$, the Harnack constant of $(\mathfrak X,\mathfrak E,\pi,\mu)$, the ellipticity constant $P_e$ and the inner-uniformity constants $\alpha, A$ such that,
for any $R>0$, $x\in U$, and $x_r$ (defined in Lemma~\ref{lem-inner-point}),
$$\max_{y\in B_U(x, r)}\{\phi_0(y)\}\le C_U \phi_0(x_r).$$
Moreover,   for any  $x\in U$ and $r\in (0,2A_1R)$, the $\pi_{\phi_0}$ volume of $B_U(x,r)$ satisfies 
$$ \sum_{y\in B_U(x,r)} \pi_{\phi_0}(y) \approx  \pi(B(x,r)) \phi_0(x_r)^2.$$
 \end{theo}
 The following estimates are derived from the properties of $\phi_0$ stated above and the geometry of inner-uniform domains. They will be useful in extracting usable formulas for $P_U(t,x,y)$.
 \begin{cor} There  are constants $a_1,A_1, A_2$, which depend only on the Harnack constant of $(\mathfrak X,\mathfrak E,\pi,\mu)$ and on $(\alpha,A)$, such that for all $x,z\in U$ and $r>0$,     
$$   a_1 (1+d_U(x,z))^{-A_1} \le  \frac{\phi_0(x)}{\phi_0(z)}    \le  A_1(1+ d_U(x,z))^{A_1},$$
and, whenever $d_U(x,z)\le A_2r$   and $0<s<r$,
$$a_1  \le   \frac{\phi_0(x_r)}{\phi_0(x_s)}  \le  A_1\left( \frac{r}{s}\right)^{A_1}\;\;\mbox{ and }\;\;\;a_1 \le \frac{\phi_0(x_r)}{\phi_0(z_r)}\le A_1.$$
 \end{cor}
 \begin{rem}The following useful estimate can be derived from this corollary. There is a constant $A_1'>0$ such that, for any $x,z\in U$ and $0<r\le d_U(x,z)$,
$$ \frac{\phi_0(x_r)}{\phi_0(z_r)}\le A'_1 \left(\frac{d_U(x,z)}{r}\right)^{A'_1} .$$
 \end{rem}
 These statements are proved using the properties of $\phi_0$, the inner-uniformity of $U$ and chains of Harnack balls for $\phi_0$  in $U$.
 
 \subsection{The tale of three boundaries}
 
 \begin{figure}[h] 
\begin{center}
\begin{tikzpicture}[scale=.15]
\draw [help lines] (0,0) grid (40,40);
\foreach \x in {0, 1,2,3,4,5,6,7,8}
\draw [fill, blue]  (21 +\x,20+\x) circle  [radius=.2];
\foreach \x in {0,1,2,3,4,5,6,7}
\draw [fill, blue]  (20 +\x,19-\x) circle  [radius=.2];
\foreach \x in {0,1,2,3,4,5,6,7,8}
\draw [fill, blue]  (19 -\x,19) circle  [radius=.2];
\foreach \x in {0,1,2,3,4,5,6,7,8,9,10,11,12,13,14,15,16,17,18,19,20,21,22,23,24,25,26,27,28,29,30}
\draw [fill, blue]  (5+\x,35) circle  [radius=.2];
\foreach \x in {0,1,2,3,4,5,6,7,8,9,10,11,12,13,14,15,16,17,18,19,20,21,22,23,24,25,26,27,28,29,30}
\draw [fill, blue]  (35,35-\x) circle  [radius=.2];
\foreach \x in {0,1,2,3,4,5,6,7,8,9,10,11,12,13,14,15,16,17,18,19,20,21,22,23,24,25,26,27,28,29,30}
\draw [fill, blue]  (35-\x, 5) circle  [radius=.2];
\foreach \x in {0,1,2,3,4,5,6,7,8,9,10,11,12,13,14,15,16,17,18,19,20,21,22,23,24,25,26,27,28,29,30}
\draw [fill, blue]  (5, 5+\x) circle  [radius=.2];

\end{tikzpicture}
\caption{A domain $U$, where the blue dots indicate absorbing boundary points. Consider the central point, where the three interior absorbing lines meet. To study the probability that a random walk is absorbed at the central point, we need to consider the three very different types of paths it could have taken: from above, below, or the right. Can we define an alternative notion of the boundary of $U$ that resolves this problem?}\label{fig-B*}
\end{center}
\end{figure}
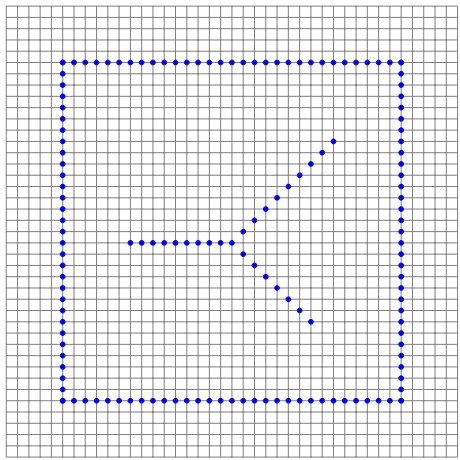
 
Before providing a deeper exploration of the exit positions, it is useful to take a look at the intrinsic boundary of a finite domain $U$.  So far we have taken the point of view that
 the boundary of $U,$ $\partial U,$ is defined as the set of those points $y$ in the the ambient space $\mathfrak X$ such that there is at least one edge $\{x,y\}\in \mathfrak E$ with $x\in U$.  We also extended the intrinsic distance $d_U$ so as to define $d_U(x,y)$ when $x\in U$ and $y\in \partial U$ by setting $d_U(x,y)=\min\{1+d_U(x,z): z\in U, \{z,y\}\in \mathfrak E\}$. The attentive reader will have noticed that this does not define a distance on $U\cup \partial U$, in general, even after setting $d_U(x,y)= \min\{1+d(z,y): z\in U\}$ for $x,y\in \partial U$.  This is because a given point on $\partial U$ may be approachable from within $U$  through several very distinct directions.    See  Figure \ref{fig-B*}.

 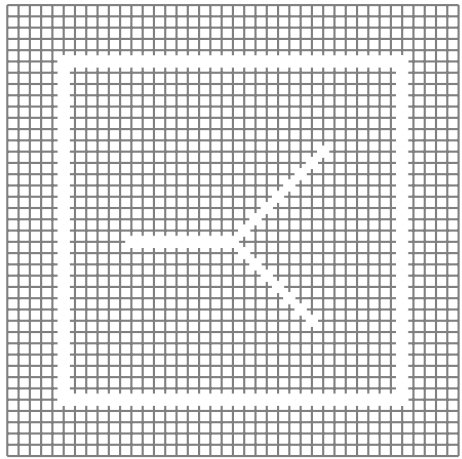
\begin{figure}[h] 
\begin{center}
\begin{tikzpicture}[scale=.15]
\draw [help lines, thick] (0,0) grid (40,40);
\foreach \x in {0,1,2,3,4,5,6,7}
\draw [fill, white]  (21 +\x,20+\x) circle  [radius=.5];
\foreach \x in {0,1,2,3,4,5,6,7}
\draw [fill, white]  (20 +\x,19-\x) circle  [radius=.5];
\foreach \x in {0,1,2,3,4,5,6,7,8}
\draw [fill, white]  (19 -\x,19) circle  [radius=.5];
\foreach \x in {0,1,2,3,4,5,6,7,8,9,10,11,12,13,14,15,16,17,18,19,20,21,22,23,24,25,26,27,28,29,30}
\draw [fill, white]  (5+\x,35) circle  [radius=.5];
\foreach \x in {0,1,2,3,4,5,6,7,8,9,10,11,12,13,14,15,16,17,18,19,20,21,22,23,24,25,26,27,28,29,30}
\draw [fill, white]  (35,35-\x) circle  [radius=.5];
\foreach \x in {0,1,2,3,4,5,6,7,8,9,10,11,12,13,14,15,16,17,18,19,20,21,22,23,24,25,26,27,28,29,30}
\draw [fill, white]  (35-\x, 5) circle  [radius=.5];
\foreach \x in {0,1,2,3,4,5,6,7,8,9,10,11,12,13,14,15,16,17,18,19,20,21,22,23,24,25,26,27,28,29,30}
\draw [fill, white]  (5, 5+\x) circle  [radius=.5];

\end{tikzpicture}
\caption{The extended boundary $\partial^*U$ defined by dangling edges. Note that the central point has three dangling edges pointing toward it, indicating that three steps that a random walk could take at the time it's absorbed.}\label{fig-Bd}
\end{center}
\end{figure}

 It is useful to introduce  the {\em extended boundary}, $\partial^*U$ of $U.$  See Figure~\ref{fig-Bd}. To justify this definition, think of the cable graph, which is a continuous analog of $(\mathfrak X, \mathfrak E)$ where the edges from $\mathfrak E$ are replaced by unit segments.
 Now, when considering the domain $U$, keep all the edges between any two points in $U$ (that is the set $\mathfrak E_U=\mathfrak E\cap (U\times U)$) but keep also the dangling half-edges $\{x,y\}$, $x\in U$, $y\in \mathfrak X\setminus U$ each of which carry a edge weight $\mu_{xy}$. Each of these so called dangling half-edge defines a distinct boundary point in 
 ${\partial^*U =\{y^*_x=\{x,y\}: x\in U, y\in \mathfrak X\setminus U\}}$.  In some sense, this is the largest natural boundary we can associate to $U$ viewed as a domain in $(\mathfrak X,\mathfrak E)$.   By using this boundary we can record not only the exit point $y$ but also the point $x$ representing the position in $U$ from which the exit occurred.   Now, it is clear that the space $U^*=U\cup \partial^*U$ can be equipped with a metric $d_U$ that extends the inner metric defined on $U$ in a natural way.  Here we think of each dangling edge as a unit interval open on one end and we close that interval by adding the missing boundary point named $y^*_x=\{x,y\}$

Each extended boundary point is attached to exactly one vertex $x \in U$ and each original boundary point $y\in \partial U$ corresponds to a finite collection of extended boundary points $\{y^*_x: x\in \nu(y)\}$ parametrized by the set we call $\nu(y)$ (see Section \ref{sec-BC}).   

Here, we are mostly interested in the original boundary and the extended boundary serves as a useful tool in studying the harmonic measure and Poisson kernel for $U$.  Nevertheless, we should also mention the 
intrinsic boundary $\partial^\bullet U$ which is associated with the data  $(U,\mathfrak E_U, \pi |_U, \mu|_{\mathfrak E_U}, K_U)$. See Figure \ref{fig-Bb}. This data  suffices to tell which points in $U$ have at least one neighbor in $\mathfrak X\setminus U$,  because at such a point $x$, $\sum_yK_U(x,y)<1$. But it retains no information about the individual dangling edges and their respective weights. For any point $x\in U$ such that $\sum_yK_U(x,y)<1$, we introduce 
an abstract boundary point $x^\bullet$ which we may think of as a cemetery point attached to $x$.  Each of the abstract boundary points $x^\bullet$  is attached to $x$ by an abstract  boundary edge $\{x,x^\bullet\}$ so that the new graph $$(U\cup \partial^\bullet U, \mathfrak E^\bullet_U), \;\;\mathfrak E^\bullet_U=\mathfrak E_U\cup \{\{x,x^\bullet\}: x^\bullet \in \partial ^\bullet U\},$$ 
is a connected graph with subgraph $(U,\mathfrak E_U)$.   It is possible to construct the intrinsic boundary $\partial ^\bullet U$ from the extended boundary $\partial ^*U$. Namely, each point ${x^\bullet\in \partial ^\bullet U}$
corresponds to the collection $\{y _x^*=\{x,y\}: y\in \partial U\}$ of extended boundary points.  The edge $(x,x^\bullet)$ in $\mathfrak E_U^\bullet$ can be given the weight $\sum_{y: y^*_x=\{x,y\}} \mu_{xy}$.  
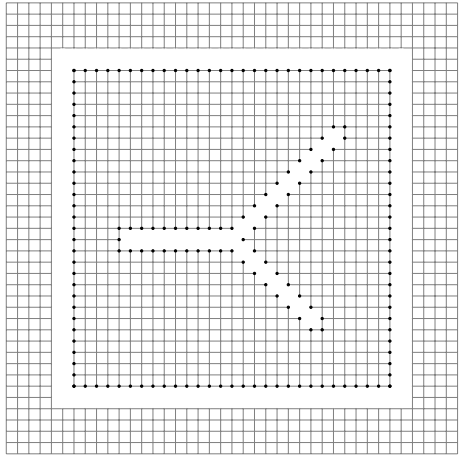
\begin{figure}[h] 
\begin{center}
\begin{tikzpicture}[scale=.15]
\draw [help lines ] (0,0) grid (40,40);
\foreach \x in {0,1,2,3,4,5,6,7,8}
\draw [fill, white]  (21+\x,20+\x) circle  [radius=.9];
\foreach \x in {0,1,2,3,4,5,6,7}
\draw [fill, white]  (20 +\x,19-\x) circle  [radius=.9];
\foreach \x in {0,1,2,3,4,5,6,7,8}
\draw [fill, white]  (19 -\x,19) circle  [radius=.9];
\foreach \x in {0,1,2,3,4,5,6,7,8,9,10,11,12,13,14,15,16,17,18,19,20,21,22,23,24,25,26,27,28,29,30}
\draw [fill, white]  (5+\x,35) circle  [radius=.9];
\foreach \x in {0,1,2,3,4,5,6,7,8,9,10,11,12,13,14,15,16,17,18,19,20,21,22,23,24,25,26,27,28,29,30}
\draw [fill, white]  (35,35-\x) circle  [radius=.9];
\foreach \x in {0,1,2,3,4,5,6,7,8,9,10,11,12,13,14,15,16,17,18,19,20,21,22,23,24,25,26,27,28,29,30}
\draw [fill, white]  (35-\x, 5) circle  [radius=.9];
\foreach \x in {0,1,2,3,4,5,6,7,8,9,10,11,12,13,14,15,16,17,18,19,20,21,22,23,24,25,26,27,28,29,30}
\draw [fill, white]  (5, 5+\x) circle  [radius=.9];


\draw [fill]  (20,20) circle [radius =.1];

\foreach \x in {0,1,2,3,4,5,6,7,8}
\draw [fill]  (22+\x,20+\x) circle  [radius=.1];
\foreach \x in {1,2,3,4,5,6,7,8,9}
\draw [fill]  (20+\x,20+\x) circle  [radius=.1];
\draw [fill]  (30,29) circle  [radius=.1];

\foreach \x in {0,1,2,3,4,5,6,7}
\draw [fill]  (21 +\x,19-\x) circle  [radius=.1];
\foreach \x in {0,1,2,3,4,5,6,7}
\draw [fill]  (20 +\x,18-\x) circle  [radius=.1];
\draw [fill]  (28,11) circle  [radius=.1];

\foreach \x in {0,1,2,3,4,5,6,7,8,9}
\draw [fill]  (19 -\x,20) circle  [radius=.1];
\foreach \x in {0,1,2,3,4,5,6,7,8,9}
\draw [fill]  (19 -\x,18) circle  [radius=.1];
\draw [fill]  (10,19) circle  [radius=.1];

\foreach \x in {1,2,3,4,5,6,7,8,9,10,11,12,13,14,15,16,17,18,19,20,21,22,23,24,25,26,27,28,29}
\draw [fill]  (5+\x,34) circle  [radius=.1];
\foreach \x in {1,2,3,4,5,6,7,8,9,10,11,12,13,14,15,16,17,18,19,20,21,22,23,24,25,26,27,28,29}
\draw [fill]  (34,35-\x) circle  [radius=.1];
\foreach \x in {1,2,3,4,5,6,7,8,9,10,11,12,13,14,15,16,17,18,19,20,21,22,23,24,25,26,27,28,29}
\draw [fill]  (35-\x, 6) circle  [radius=.1];
\foreach \x in {1,2,3,4,5,6,7,8,9,10,11,12,13,14,15,16,17,18,19,20,21,22,23,24,25,26,27,28,29}
\draw [fill]  (6, 5+\x) circle  [radius=.1];

\end{tikzpicture}
\caption{The intrinsic boundary $\partial^\bullet U$. The marked dots correspond to those points $x \in U$ to which an abstract boundary point $x^\bullet$ is attached. The attached abstract boundary points $x^\bullet$ are not shown explicitly. In this case, the central point is gone.}\label{fig-Bb}
\end{center}
\end{figure}  

Finally, we note that, in general,  there is no good direct relation between the natural boundary $\partial U$ and the intrinsic boundary $\partial^\bullet U$. Each of them can be seen as a different contraction of the extended boundary $\partial ^*U$.

 \subsection{Hitting probabilities for the extended boundary $\partial^*U$}\label{sec-hitting}
 We now explain some of the consequences of the theorems of Section~\ref{subsec-inner-uniform} on $P_U(t,x,y)$ within an inner-uniform domain $U$. The first thing to note is that Theorem \ref{theo-PUo} applies, uniformly, to all finite inner $(\alpha,A)$-uniform domains 
 in a given underlying structure $(\mathfrak X,\mathfrak E,\pi,\mu)$ that is a $2$-Harnack graph.  In order to get a more complete result which allows for varying a starting point and a fixed time horizon $t$ (Theorem \ref{theo-PUo} gives a two-sided estimate only for $P_U(o,y)$), we need to  estimate (see Theorem \ref{th-Pab2})
 $$P_U(t,x,y)=  \phi_0(x) \sum_{\ell=0}^{t-1} \beta_0^\ell \sum_{z\in \nu(y)} \phi_0(z)k^\ell_{\phi_0}(x,z)\mu_{zy}    .$$

 It is easier and more informative to first consider this question  in terms of the extended boundary $\partial^*U$ and that is how we now  proceed.   Any point $y^*_z=\{z,y\}\in \mathfrak E\cap U\times \partial U=\partial^*U$ can be reached only from the point $z=z_y$. We set
 $$p_U(t,x,y^*_z)= \sum_0^{t-1}k^\ell_U(x,z)\mu_{zy}/\pi(y)
 =\phi_0(x)\phi_0(z)\sum_0^{t-1} \beta_0^{\ell}k_{\phi_0}^{\ell}(x,z)\mu_{zy}/\pi(y)$$
 so that
 $$p_U(t,x,y)=\sum_{z\in \nu(y)} p_U(t,x,y^*_z).$$
 
 The quantity   $p_U(t,x,y^*_z)$  is equal to $0$ unless $t\ge 1+d_U(x,z)$   and we write
 $$ p_U(t,x,y^*_z)=  \phi_0(x) \phi_0(z)\sum_{\ell=d_U(x,y)-1}^{t-1} \beta_0^\ell  k^\ell_{\phi_0}(x,z)\mu_{zy} /\pi(y)    .$$

For clarity, we split  the problem into several cases (represented in the next four lemmas) even though these different cases can be captured by one final estimate, Theorem~\ref{thm-global}. The exponential term in the estimate on $P_U(t,x,y^*_z)$ depends on $t$ and $d_U(x,z)$. The lemmas distinguish between four different domains (depending on $t$ and $d_U(x,z)$ and with some non-empty intersections), and highlight the different behavior of the exponential term in the estimate for $P_U(t,x,y^*_z)$ within each of these domains. In Lemma~\ref{lem-Z1}, the exponential term plays an important role in the estimate; in Lemma~\ref{lem-Z2}, the exponential term is still there, but less important; and in Lemmas~\ref{lem-Z3} and~\ref{lem-Z4}, the exponential term disappears.

All four of the following lemmas (Lemmas~\ref{lem-Z1},~\ref{lem-Z2},~\ref{lem-Z3}, and~\ref{lem-Z4}) take place under the assumptions of Theorem~\ref{th-C}: $(\mathfrak X, \mathfrak E, \pi, \mu)$ is a 2-Harnack graph satisfying the ellipticity condition~\eqref{elliptic} and $U\subseteq \mathfrak X$ is a finite inner $(\alpha,A)$-uniform domain with Perron-Frobenius eigenfunction $\phi_0$. Observe that, by construction  and because of the ellipticity assumption,
$$P_e^{-1}\le \frac{\mu_{uv}}{\pi(v)}\le 1.$$

\begin{lem}[$1+d_U(x,z_y)\le t \leq (1+d_U(x,z))^{2-\epsilon}$]  \label{lem-Z1}
Under the assumptions of Theorem~\ref{th-C}, fix $\epsilon>0$ and assume that $x\in U,y^*_z\in \partial ^*U$ and $t$ are such that $1+d_U(x,z)\le t\le (1+d_U(x,z))^{2-\epsilon}$, $z=z_y$.  Then
$$\frac{e^{-C_1 d_U(x,z)^2/t}\mu_{zy}}{\pi(B(x,\sqrt{t}))} \le P_U(t,x,y_z^*) \le\frac{e^{-c_1 d_U(x,z)^2/t}\mu_{zy}}{\pi(B(x,\sqrt{t}))}.$$
\end{lem}

\begin{proof}  If $t=1$, we must have $x=z$ and it follows that $$P_U(1,x,y^*_z) =K(x,y)=\mu_{xy}/\pi(x)\approx 1\approx \pi(y)/\pi(x)$$ by the ellipticity assumption. In what follows, we assume that $t>1$. 

Recall that, the hypotheses imply that  $\beta_0\ge  1 -C/R^2$. Because 
$$t< d^2(x,y) \le (A_1 R)^2,$$ we can ignore the factors $\beta_0^{\ell}$ for $\ell\le t$ because they are roughly constant.  
It now suffices to bound
$$\sum_{\ell=d_U(x,z)}^{t-1} k^\ell_{\phi_0}(x,z).$$

For the upper bound, Theorem~\ref{theo-PHItheta} gives (with constants $c,C$ changing from line to line)
\begin{eqnarray*} \lefteqn{\phi_0(z) \sum_{\ell=d_U(x,z)}^{t-1} k^\ell_{\phi_0}(x,z) }&& \\
& \le & \frac{C\phi_0(z)}{ \pi (\phi_0^2\mathbf 1_{B_U(x,\sqrt{t})}) }
\sum_{\ell=d_U(x,z)}^{t-1}  \frac{\pi (\phi_0^2\mathbf 1_{B_U(x,\sqrt{t})})}{\pi(\phi_0^2\mathbf 1_{B_U(x,\sqrt{\ell})})}
e^{-cd_U(x,z)^2/\ell}\\
&\le & \frac{C\phi_0(z)}{\phi_0(x_{\sqrt{t}})^2\pi(B(x,\sqrt{t}))}  \sum_{\ell=d_U(x,z)}^{t-1}  (t/\ell)^\kappa 
e^{-cd_U(x,z)^2/\ell}\\
&\le &  \frac{C \phi_0(z)d_U(x,z)^2}{\phi_0(x_{\sqrt{t}})^2\pi(B(x,\sqrt{t}))}
e^{-cd_U(x,z)^2/t}\\
&\le &  \frac{C\phi_0(x)}{\phi_0(x_{\sqrt{t}})^2\pi(B(x,\sqrt{t}))}
e^{-cd_U(x,z)^2/t}\\
&\le & \frac{C}{\pi(B(x,\sqrt{t}))}
e^{-cd_U(x,z)^2/t}\\
\end{eqnarray*}

 The lower bound follows by similar computations and estimates. The only tricky part is that we only have a heat kernel lower bound on the sum $k^{\ell}_{\phi_0}+k^{\ell+1}_{\phi_0}$. This is perfectly suited for the desired result, except when  $t=1+d_U(x,z)$, in which case the sum $\sum_{\ell=d_U(x,z)}^{t-1}k_{\phi_0}^{\ell} (x,z)$ contains exactly one term. This case is handled by direct inspection and using the ellipticity hypothesis. 
 \end{proof}
 
\begin{lem}[$1+d_U(x,z_y)\le t \le A_2(1+ d_U(x,z_y))^2$]  \label{lem-Z2}
Under the assumptions of Theorem~\ref{th-C}, fix $A_2$  and assume that $x\in U,y^*_z\in \partial U$ and $t$ are such that $1+d_U(x,z)\le t \le A_2(1+ d_U(x,z))^2$, $z=z_y$.  Then
$$\frac{c_1t\phi_0(x)\phi_0(z)e^{-C_1d_U(x,z)^2/t} \mu_{zy}}{\phi_0(x_{\sqrt{t}})^2\pi(B(x,\sqrt{t}))}\le P_U(t,x,y^*_z)\le\frac{C_1t\phi_0(x) \phi_0(z) e^{-c_1d_U(x,z)^2/t}\mu_{zy}}{\phi_0(x_{\sqrt{t}})^2\pi(B(x,\sqrt{t}))} .$$
\end{lem} 
 \begin{proof}  Write  (with constants $c,C$ changing from line to line)
\begin{eqnarray*} \lefteqn{\phi_0(z) \sum_{\ell=d_U(x,z)}^{t-1} k^\ell_{\phi_0}(x,z) }&& \\
& \le & \frac{C\phi_0(z)}{ \pi (\phi_0^2\mathbf 1_{B_U(x,\sqrt{t})}) }
\sum_{\ell=d_U(x,z)}^{t-1}  \frac{\pi (\phi_0^2\mathbf 1_{B_U(x,\sqrt{t})})}{\pi(\phi_0^2\mathbf 1_{B_U(x,\sqrt{\ell})})}
e^{-cd_U(x,z)^2/\ell}\\
&\le & \frac{C\phi_0(z)}{\phi_0(x_{\sqrt{t}})^2\pi(B(x,\sqrt{t}))}  \sum_{\ell=d_U(x,z)}^{t-1}  (t/\ell)^\kappa 
e^{-cd_U(x,z)^2/\ell}\\
&\le & \frac{C \phi_0(z)t}{\phi_0(x_{\sqrt{t}})^2\pi(B(x,\sqrt{t}))}
e^{-cd_U(x,z)^2/t}.
\end{eqnarray*}
 A matching lower bound follows similarly. 
  \end{proof}
 
 \begin{lem}[$1+d_U(x,z_y)^2\le t \le A_2R^2$]  \label{lem-Z3}
Under the assumptions of Theorem~\ref{th-C}, fix $A_2$  and assume that $x\in U,y^*_z\in \partial U$ and $t$ are such that $(1+d_U(x,z))^2\le t \le A_2 R^2$, $z=z_y$.  Then, setting $d=d_U(x,z)$, 
$$P_U(t,x,y^*_z)\approx  \frac{(1+d^2)\phi_0(x) \phi_0(z)\mu_{zy}}{\phi_0(x_{d})^2\pi(B(x,d))}\left\{1+ \frac{1}{1+d^2}\sum_{\ell=d^2}^{t} 
\frac{\phi_0(x_{d})^2\pi(B(x,d))}{\phi_0(x_{\sqrt{l}})^2\pi(B(x,\sqrt{\ell}))} \right\}.$$
\end{lem} 
\begin{proof} This is clear based on the proof of the previous estimate. \end{proof} 

\begin{defin}\label{defn-h-func}  Let $T_U$ be such that $\beta_0=1-1/T_U$. For $x\in U,y^*_z\in \partial U$ and $t\ge d(x,z)^2$, set $d=d_U(x,z)$, $V(x,d)=\pi(B(x,d))$ and 
$$H(t,x,z)= 1+ \left\{\begin{array}{cl}  0   &\mbox{ for  } 1+d\le t  <d^2\\
 \frac{\phi_0(x_{d})^2V(x,d)}{1+d^2}\sum_{l=d^2}^{t} 
\frac{1}{\phi_0(x_{\sqrt{l}})^2\pi(B(x,\sqrt{\ell}))} & \mbox{ for } d^2\le t\le R^2\\  
 H(R^2,x,z)+ \frac{\phi_0(x_{d})^2V(x,d)}{1+d^2}\frac{(\min\{t,T_U\}-R^2)_+}{\phi_0(o)^2\pi(U)}& \mbox{ for }   R^2\le t . \end{array}\right. $$
\end{defin}

 \begin{lem}[$1+d(x,z_y)^2\le  t $]  \label{lem-Z4}  Under the assumptions of Theorem~\ref{th-C}, let $T_U$ be such that $\beta_0=1-1/T_U$.   For all
  $x\in U,y^*_z\in \partial U$ and $t\ge d_U(x,z)^2$, $z=z_y$,  
$$P_U(t,x,y^*_z)\approx  \frac{(1+d_U(x,z)^2)\phi_0(x) \phi_0(z)\mu_{zy}}{\phi_0(x_{d_U(x,z)})^2\pi(B(x,d_U(x,z)))}H(t,x,z).$$
\end{lem} 

\begin{rem}  In many cases (e.g., for any finite inner-uniform domain $U$ in $\mathbb Z^n$, $n\neq 2$, and many particular examples in $\mathbb Z^2$),   we automatically have
$$\forall\, t\in [2d^2,R^2] ,\;\; \sum_{\ell=d^2}^{t} 
\frac{1}{\phi_0(x_{\sqrt{l}})^2\pi(B(x,\sqrt{\ell}))} \approx    \frac{1+d^2}{\phi_0(x_{d})^2\pi(B(x,d))}, $$
where $d=d_U(x,z)$. In such cases,  the function $H(x,t)$ satisfies
$$H(t,x,z)\approx  \left\{\begin{array}{cl}  1 & \mbox{ for }  d^2\le t\le R^2\\
1+ \frac{\phi_0(x_{d})^2V(x,d)}{1+d^2}\frac{(\min\{t,T_U\}-R^2)_+}{\phi_0(o)^2\pi(U)}& \mbox{ for }   R^2\le t . \end{array}\right. $$
and Lemma \ref{lem-Z3} simplifies to give
$$P_U(t,x,y^*_z)\approx  \frac{(1+d_U(x,z)^2)\phi_0(x) \phi_0(z)\mu_{zy}}{\phi_0(x_{d_U(x,z)})^2\pi(B(x,d_U(x,z)))} $$ 
for  $d_U(x,z)^2\le t\le A_2R^2$. \end{rem}

The following theorem is proved by inspection of the different cases described above.  We also use the fact that, for any $\kappa\in \mathbb R$ and  $\omega>0$ there exists $0<c\le C<+\infty$ such that, for all $ 0<t<d^2$,
$$ce^{-2\omega d^2/t}\le \left(\frac{d^2}{t}\right)^{\kappa } e^{- \omega d^2/t}\le  C e^{-(\omega/2) d^2/t}.$$
\begin{theo}[Global estimate of $P_U(t,x,y_z^*)$]\label{thm-global}   
Under the assumptions of Theorem~\ref{th-C}, for all $x\in U,y^*_z\in \partial^*U$, $z=z_y\in U$, $d=d_U(x,z)$ and $t\ge 1+d$, with $H(t,x,z)$ from Definition~\ref{defn-h-func}, the hitting probability of $y_z^*$ before time $t$ for the chain started at $x$ is bounded above and below by
expressions of the form
$$c_1  \frac{(1+d^2)\phi_0(x) \phi_0(z)\mu_{zy}}{\phi_0(x_{d})^2\pi(B(x,d))} H(t,x,z) e^{-c_2 d^2/t}$$
where the constants $c_1,c_2$ differ in the lower bound and in the upper bound. These constants depend only on  the Harnack constant of $(\mathfrak X,\mathfrak E,\pi,\mu)$, the ellipticity constant $P_e$ and the inner-uniformity constants $\alpha, A$ of $U$ .\end{theo}

We conclude this section with two more statements. The first concerns the central point $o$ and gives a two-sided estimate for $P_U(t,o,y_z^*)$ that holds for all $t\ge d_U(o,y^*_z)$ and all extended boundary points $y^*_z$.
The second gives a two-sided estimate for the harmonic measure   $P_U(x,y^*_z)$  that holds for all  $x\in U$, $y_z^*\in \partial^*_U$.  

\begin{theo}[Hitting probabilities from the central point $o$]  Fix  $\alpha\in (0,1]$ and $A\ge 1$.
Assume that $(\mathfrak X,\mathfrak E,\pi,\mu)$ is a $2$-Harnack graph satisfying the ellipticity condition {\em (\ref{elliptic})} and that $U$ is a finite inner $(\alpha,A)$-uniform domain with Perron-Frobenius  eigenvalue and  eigenfunction
$\beta_0,\phi_0$ with $\pi(\phi_0^2)=1$ and set $T_U=(1-\beta_0)^{-1}$.
There are constants $c,C\in (0,\infty)$ depending only on the Harnack constant of $(\mathfrak X,\mathfrak E,\pi,\mu)$, the ellipticity constant $P_e$, and the inner-uniformity constants $\alpha, A$ of $U$ such that,  for all $t>0$ and $y_z^*\in \partial^*U$,
$$c \min\{t,T_U\}\frac{ \mu_{zy}  \phi_0(z) }{\sqrt{\pi(U)}}  e^{-CR^2/t}\le P_U(t,o,y^*_z) \le  C \min\{t,T_U\}\frac{ \mu_{zy}  \phi_0(z) }{\sqrt{\pi(U)}}  e^{-cR^2/t}. $$
\end{theo}
\begin{theo}[Harmonic measure from an arbitrary starting point] \label{theo-HMarb}
 Fix  $\alpha\in (0,1]$ and $A\ge 1$.
Assume that $(\mathfrak X,\mathfrak E,\pi,\mu)$ is a $2$-Harnack graph satisfying the ellipticity condition {\em (\ref{elliptic})} and that $U$ is a finite inner $(\alpha,A)$-uniform domain with Perron-Frobenius  eigenvalue and  eigenfunction
$\beta_0,\phi_0$ with $\pi(\phi_0^2)=1$ and set $T_U=(1-\beta_0)^{-1}$.
There are constants $c,C\in (0,\infty)$ depending only on the Harnack constant of $(\mathfrak X,\mathfrak E,\pi,\mu)$, the ellipticity constant $P_e$ and the inner-uniformity constants $\alpha, A$ of $U$ such that,  for all $x\in U$ and  $y_z^*\in \partial^*U$,
$$P_U(x,y^*_z)  \approx \phi_0(x)\phi_0(z)\mu_{zy}\left\{ T_U+\sum_{d_U(x,z)^2}^{R^2} \frac{1}{\phi_0(x_{\sqrt{\ell}})^2 \pi(B(x,\sqrt{\ell}))} \right\}     .        $$
\end{theo}

\begin{lem} Assume that the function $V: (0,N]\ra (0,\infty)$  satisfies $$V(2r)\le CV(r),\;\;  V(s)\le CV(r)$$ and
\begin{equation}\label{>2}
\frac{V(r)}{V(s)}\ge c\left(\frac{r}{s}\right)^{2+\epsilon} \end{equation}
for all $1\le s<r\le N$. Then we have
$$\forall\,d\in (1,N/2),\;\; \sum_{\ell=d^2}^{N^2}\frac{1}{V(\sqrt{\ell})} \approx  \frac{1+d^2}{V(d)}.$$
\end{lem}
\begin{proof}  Write
\begin{eqnarray*}   \sum_{\ell=d^2}^{N^2}\frac{1}{V(\sqrt{\ell})} &=& \frac{1}{V(d) }\sum_{\ell=d^2}^{N^2}\frac{V(d)}{V(\sqrt{\ell})}\\
&\le & \frac{C}{V(d) } \sum_{k=0}^{2\log_2 (N/d)} \sum_{\ell: \ell \approx  d^22^k} \left(\frac{d^2}{d^22^k}\right)^{2+\epsilon}  \\
&\approx &\frac{ C' d^2}{ V(d)} .
\end{eqnarray*}
The matching lower bound  follows from the quasi-monotonicity of  $V$ with ${d\le N/2}$ because it implies that the sum contains at least $d^2$ terms of size at least order $1/V(d)$. (We say that $A(x)$ is quasi-monotone if there exists $c>0$ such that, for all $s<r$, $A(x_s) \leq cA(x_r)$.)
\end{proof}

\begin{rem} \label{rem-easy}
The lemma is often useful in applying Theorem \ref{theo-HMarb}.  The question is whether or not the function $r\mapsto  \phi_0(x_r)^2V(x,r)$ satisfies the hypotheses of the lemma.
But remember that  $\phi_0(x_k)^2V(x,k)\approx \pi(\phi_0^2 \mathbf 1_{B(x,k)})$  and Theorems~\ref{th-H} and~\ref{th-C}  state that this function is doubling (it is obviously quasi-monotone). In fact,
this function is the product of two functions $r\mapsto \phi_0(x_r)^2$ and $r\mapsto V(x,r)$, each of which is quasi-monotone and doubling. If any one of these two functions, by itself, satisfies (\ref{>2}),
the product does also.   If say,   $V(x,r)\approx r^2$, then it suffices to establish that   $\phi_0(x_r)/\phi_0(x_s)\ge c (r/s)^\eta$ for some $\eta>0$.    In any such situation, the conclusion of Theorem \ref{theo-HMarb} simplifies to read
\begin{equation} \label{easy-P}
P_U(x,y^*_z)\approx \phi_0(x) \phi_0(z)\mu_{zy}\left\{ T_U+\frac{ 1+d_U(x,z)^2}{\phi_0(x_{d_U(x,z)})^2\pi(B(x,d_U(x,z)))}\right\}.
\end{equation}
\end{rem}

\subsection{Examples}\label{sec-ex}

\subsubsection*{Three-player gambler's ruin problem} 

 \begin{figure}[h] 
\begin{center}
\begin{tikzpicture}[scale=.35]

\foreach \x in {1,2,3,4,5,6,7,8,9,10,11,12,13,14,15,16,17,18,19,20}
{\draw  (\x,0) -- (0,\x);
\draw (\x,0) -- (\x,20-\x);
\draw (0,\x)--(20-\x,\x);
}
\draw (0,0)--(0,20);
\draw (0,0)-- (20,0);

\foreach \x in {1,2,3,4,5,6,7,8,9,10,11,12,13,14,15,16,17,18,19}
{\draw [fill, blue] (\x,0) circle [radius=.2];}

\foreach \x in {1,2,3,4,5,6,7,8,9}
{\draw  [fill, yellow] (\x,1) circle [radius=.2];}
\foreach \x in {1,2,3,4,5,6,7,8,9}
{\draw  [fill, yellow] (\x,2) circle [radius=.2];}
\foreach \x in {1,2,3,4,5,6,7,8}
{\draw  [fill, yellow] (\x,3) circle [radius=.2];}
\foreach \x in {1,2,3,4,5,6,7,8}
{\draw  [fill, yellow] (\x,4) circle [radius=.2];}
\foreach \x in {1,2,3,4,5,6,7}
{\draw  [fill, yellow] (\x,5) circle [radius=.2];}
\foreach \x in {1,2,3,4,5,6,7}
{\draw  [fill, yellow] (\x,6) circle [radius=.2];}
\foreach \x in {1,2,3,4,5,6}
{\draw  [fill, yellow] (\x,7) circle [radius=.2];}
\foreach \x in {1,2,3,4,5,6}
{\draw  [fill, yellow] (\x,8) circle [radius=.2];}
\foreach \x in {1,2,3,4,5}
{\draw  [fill, yellow] (\x,9) circle [radius=.2];}
\foreach \x in {1,2,3,4,5}
{\draw  [fill, yellow] (\x,10) circle [radius=.2];}
\foreach \x in {1,2,3,4}
{\draw  [fill, yellow] (\x,11) circle [radius=.2];}
\foreach \x in {1,2,3,4}
{\draw  [fill, yellow] (\x,12) circle [radius=.2];}
\foreach \x in {1,2,3}
{\draw  [fill, yellow] (\x,13) circle [radius=.2];}
\foreach \x in {1,2,3}
{\draw  [fill, yellow] (\x,14) circle [radius=.2];}
\foreach \x in {1,2}
{\draw  [fill, yellow] (\x,15) circle [radius=.2];}
\foreach \x in {1,2}
{\draw  [fill, yellow] (\x,16) circle [radius=.2];}
\foreach \x in {1}
{\draw  [fill, yellow] (\x,17) circle [radius=.2];}
\foreach \x in {1}
{\draw  [fill, yellow] (\x,18) circle [radius=.2];}

\end{tikzpicture}
\caption{The gambler's ruin problem with $3$ players, with starting points $x$ in yellow and exit points $y$ in blue. If we know $P_U(x,y)$ for all yellow $x$ and blue $y$, then all other possibilities $P_U(x',y')$ can be obtained by symmetry.} \label{GR3}
\end{center}
\end{figure}
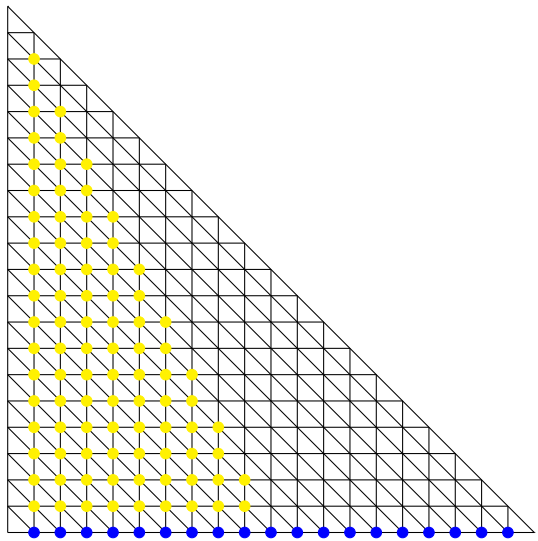

 We return to Example \ref{exa-GR}, the three-player gambler's ruin problem which evolves in the triangle $$U=\{(x_1,x_2): 0< x_1,0\le x_2, x_1+x_2< N\}.$$ In Example~\ref{exa-GR} we gave approximations to the harmonic measure starting from $N/4,N/4$. We here complete this, giving uniform estimates from any start. The natural symmetries of the problem imply that each of the three corners of the triangle are equivalent (under appropriate transformations) so we can focus on the  the corner at the origin. We will describe two-sided bounds on the harmonic measure $P_U(x,y)$
when $x=(x_1,x_2)$ with $0<x_1, 0<x_2$, $2x_1+x_2\le N$,   and $y=(y_1,0)$, $0<y_1<N$.   See Figure \ref{GR3}. In this example, $R\approx N$, $T_U\approx N^2$, $\mu_{zy}\approx 1$, $\pi(B(x,r))\approx r^2$.  
Each boundary point $y$ corresponds to either one or two
extended boundary points. For any $y$ which has two extended boundary points $\{z,y\}$ and $\{z',y\}$, the internal points $z,z'$ are neighbors in $U$. This means there is no real need to distinguish them when estimating $P_U(x,y)$.   For each $y=(y_1,0)$,  $0<y_1<N$, we set  $z_y=(y_1,1)$ and $z'=(y_1-1,1)$ with the convention that $z_{(1,0)}=z'_{(1,0)}=(1,1)$ and $z_{(N-1,0)}=z'_{(N-1,0)}=(N-2,1)$. Next, we appeal to estimate (\ref{GRphi}) to control $\phi_0$.
For $z=(z_1,1)$,  $0<z_1<N-1$, $$\phi_0(z) \approx  N^{-6} z_1^2(N-z_1)^2.$$ 
For $x=(x_1,x_2)$ with $0<x_1, 0<x_2,  2x_1+x_2\leq N$, 
$$\phi_0(x)\approx   N^{-6} x_1x_2(x_1+x_2)(N-(x_1+x_2))(N-x_2).$$
Remark \ref{rem-easy} applies to this example and we can use (\ref{easy-P}).  
Assume first that $d=d_U(x,z_y)\ge N/8$. In this case,  we have
\begin{eqnarray*}
P_U(x,y) &\approx & N^2 \phi_0(x)\phi_0(z_y)\\ & \approx & N^{-10} x_1x_2(x_1+x_2)(N-(x_1+x_2))(N-x_2)y_1^2(N-y_1)^2.\end{eqnarray*}
 Assume instead that $d=d_U(x,z_y)\le N/8$. In that case $ |x_1-z_1|+|x_2-1|\le 2d \le N/4$ and $\phi_0(x_d)\approx \phi_0((x_1+d,x_2+d)) $.  It follows that
 \begin{eqnarray*}
 P_U(x,y) &\approx &\phi_0(x)\phi_0(z_y)  \frac{(1+d^2)}{\phi_0(x_d)^2(1+d^2)}\\
 &\approx &  \frac{ x_1x_2(x_1+x_2)y_1^2 }{ (x_1+d)^2(x_2+d)^2(x_1+x_2+2d)^2  }.  \end{eqnarray*}
It is possible summarize the two case via one formula. Namely,
for all $x=(x_1,x_2)$ with $0<x_1, 0<x_2,  2x_1+x_2\le N$ and  $y=(y_1,0)$,  $0<y_1<N$, $d=d_U(x,z_y)$,
\begin{equation} \label{GRP}
P_U(x,y)\approx  \frac{x_1x_2(x_1+x_2)(N-(x_1+x_2))(N-x_2)y_1^2(N-y_1)^2}{ N^4 (x_1+d)^2(x_2+d)^2(x_1+x_2+2d)^2  }.  \end{equation}

Note that, despite appearances, $y$ appears in both the numerator and the denominator of~\eqref{GRP}. For example, if $x=(x_1,x_2)=(1,1)$ (i.e., the random walk starts in the lower left corner), then $$P_U(x,y) \approx \frac{(N-y_1)^2}{N^2y_1^4},$$ where $y=(y_1,0)$. Thus, absorption is most likely for small $y_1$ and falls off like $y_1^4$ when $y_1$ is of order $N$. Similarly, if $x=(x_1,x_2)=(1,N-2)$ (i.e., the random walk starts in the upper left corner), then $$P_U(x,y) \approx \frac{y_1^2(N-y_1)^2}{N^6},$$ where $y=(y_1,0)$. Recall from Example~\ref{exa-GR} that $$P_U(x,y) \approx \frac{y_1^2(N-y_1)^2}{N^5}$$ when $x=(x_1,x_2)=([N/4],[N/4])$ and $y=(y_1,0)$, which aligns with~\eqref{GRP}.
 
 \subsubsection{The square and cube with the center removed}
 
 Consider the cube with the center removed, $$U=\{-N,\dots,N\}^n\setminus\{(0,\dots,0)\}$$ in dimension $n\ge 2$. The boundary is 
 $$\partial U=\{(0,\dots,0)\}\bigcup\left(\bigcup_{i}^n  F_i\right),$$
 $$F_{\pm i}=\{x=(x_j)_1^n: x_j\in\{-N,\dots,N\} \mbox{ for } j\neq i; x_i=\pm (N+1)\}.$$
 Here, $\mathfrak X=\mathbb Z^n$ equipped with its natural edge set $\mathfrak E=\{\{x,y\}: \sum_1^n|x_i-y_i|=1\}$. The measure $\pi$ is counting measure and 
 we can take either $\mu_{xy}=\frac{1}{2n}$ (in which case the chain is periodic of period $2$) or an aperiodic version with $\mu_{xy}=\frac{1}{\kappa n}$, $\kappa \in (2,4)$, say.  In any of these case, $(\mathfrak X,\mathfrak E,\pi,\mu)$ is a $2$-Harnack graph and the Perron-Frobenius eigenvalue $\beta_0$ of $U$ satisfies
 $$1-\beta_0\approx  \frac{1}{N^2}.$$
This translates into $T_U\approx N^2$.  It is a bit more challenging to describe a good global two-sided estimate for the Perron-Frobenius eigenfunction $\phi_0$. The estimates differ in dimension $n\geq2$.
When $n\ge 3$ (recall the normalization $\pi(\phi_0^2)=1$), we have (see \cite{DHSZ})
\begin{eqnarray*}
\phi_0(x)&\approx_n &  \frac{1}{N^{n/2}} \left(1-\frac{1}{(1+|x|)^{n-2}}\right) \prod_1^n \left(1-\frac{|x_i|}{N+1}\right) \\
&\approx_n & \mathbf 1_{\{\mathbf 0\}}(x)\frac{1}{N^{n/2}}\prod_1^n \left(1-\frac{|x_i|}{N+1}\right).\end{eqnarray*}
In this two-sided bound, $|x|=\sum_1^n|x_i|$ and the implied constant depends on the dimension $n$.

Similarly, for $n=2$,
$$\phi_0(x)\approx  \frac{1}{N}\left(1-\frac{|x_1|}{N+1}\right) \left(1-\frac{|x_2|}{N+1}\right) \frac{\log (1+|x|)}{\log (1+N)}.$$
 
 We now use these estimates to state two-sided bounds for $P_U(x,y)$ for~${x\in U}$, ${y\in \partial_U}$.  We can let $x$ be arbitrary in $U$ and assume that $y$ belongs either to the top face $F_n=\{y=(\bar{y},N+1): \bar{y}\in \{-N,\dots,N\}^{n-1}\}$, or is equal to the central point $\mathbf 0=(0,\dots,0)$.
 
 When $n\ge 3$, Remark \ref{rem-easy} applies  and Theorem \ref{theo-HMarb}
 gives (see (\ref{easy-P})
 $$
P_U(x,y^*_z)\approx \phi_0(x) \phi_0(z)\left\{ N^2+\frac{1+|x-z|^2}{\phi_0(x_{|x-z|})^2(1+|x-z|^n)}\right\}.
$$
At $y=\mathbf 0$ and for each of its $2n$ neighbors $z$ with all coordinates zero except one equal to $\pm 1$  (recall that the point $x$ is in $U=\{-N,\dots,N\}^n\setminus\{\mathbf 0\}$),
$$P_U(x,\mathbf 0^*_z)\approx   \phi_0(x)N^{n/2} |x|^{2-n}\approx  \prod_1^n \left(1-\frac{|x_i|}{N+1}\right) |x|^{2-n}.  $$
At a point $y$ on the top face $F_n$, there is a unique neighbor $z$ of $y$ lying in $U$ and 
$$P_U(x,y)\approx  \frac{\prod_1^n \left(1-\frac{|x_i|}{N+1}\right) \prod_1^{n-1} \left(1-\frac{|y_i|}{N+1}\right)}{(N+1)
\prod_1^n \left(1-\frac{|x_i|-|x-y|}{N+1}\right)^2} |x-y|^{2-n}.$$

As an illustrative example, consider the case when  $k$ of the coordinates of $x$ are equal to  $N+1-r$,  $\ell$ of the first $n-1$ coordinates of $y$ are equal to $N$  (by assumption $y_n=N+1$), the remaining coordinates of $x$ and $y$ are less than $N/2$ and  $|x-y|$ is greater than $N/2$.    For such a configuration,
$$P_U(x,y)\approx  \left(\frac{1}{N+1}\right)^{n-1 +\ell}\left(\frac{r}{N+1}\right)^k .$$

In the case $n=2$,  we need to understand the quantity
$$S(x,d)= \sum_{\ell=d^2}^{8N^2} \frac{1}{\phi_0(x_{\sqrt{\ell}})^2(1+ \ell)},$$ where $d=d_U(x,z)$. When $d\ge N/4$,  $S(x,d)\approx   N^2$.  When $d<N/4$ and $z $ is a neighbor  of $\mathbf 0$,  
$$S(x,d) \approx (N\log N)^2\sum_{d^2}^{8N^2} \frac{1}{\ell( \log \ell )^2}\approx  (N \log N)^2  \frac{1+\log (1+2N/d) }{(1+\log N )(1+\log (1+d))} .$$
When $0\le d<N/4$ and $y$ is on one of the four  faces $F_{\pm i}$, $i=1,2$, we have $ |x_1-y_1|\le N/4$, $|x_2-y_2|\le N/4$ and this implies $|x|\ge N/2$.  Since one of $y_1,y_2$ equals $\pm (N+1)$, it follows that 
one of $|x_i-y_i|$ equals $N+1 -|x_i| $ which must be less than $d+1$. Now,  for $\ell\ge d^2$, we have
$$\frac{\phi_0(x_{\sqrt{\ell}})}{\phi_0(x_d)} \approx \frac {(N+1-|x_1| +\sqrt{\ell})(N+1-|x_2|+\sqrt{\ell})}{(N+1-|x_1|+d)(N+1-|x_2|+ d)}\ge \frac{1}{2} \frac{1+\sqrt{\ell}}{1+d}.$$
Indeed, assume for instance that for $i=1$, $ N+1-|x_1| \le d+1$. Then, for $\sqrt{\ell}\ge d\ge N-|x_1|$,  
\begin{eqnarray*}
 \frac {(N+1-|x_1| +\sqrt{\ell})(N+1-|x_2|+\sqrt{\ell})}{(N+1-|x_1|+d)(N+1-|x_2|+ d)}&\ge & \frac {N+1-|x_1| +\sqrt{\ell}}{N+1-|x_1|+d}  \\
 &\ge  &  \frac{1+\sqrt{\ell}}{2(1+d)}.\end{eqnarray*}
Now write
\begin{eqnarray*}S(x,d) &= &  \frac{1}{(1+d)^2 \phi_0(x_d)^2}\sum_{\ell=d^2}^{8N^2} \frac{ \phi_0(x_d)^2(1+d)^2}{\phi_0(x_{\sqrt{\ell}})^2(1+\ell)} \\
&\le &   \frac{C}{(1+d)^2 \phi_0(x_d)^2} \sum_{d^2}^{8N^2} \left(\frac{ 1+ d}{1+\sqrt{\ell}}\right)^{4}\\
&\approx& \frac{C}{ \phi_0(x_d)^2}.
\end{eqnarray*}
The conclusion is that, when $y=\mathbf 0$,
$$S(x,d) \approx   N^2 \log N \frac{1+\log (1+2N/d) }{1+\log (1+d)} $$
and
$$P_U(x,\mathbf 0)\approx  \left(1-\frac{|x_1|}{N+1}\right) \left(1-\frac{|x_2|}{N+1}\right) \frac{1+\log (1+2N/|x|) }{(1+\log N)(1+\log (1+|x|))}.$$
When $y$ is on $F_{\pm i}$, $i=1,2$,
whereas for $y$ on one of the faces $F_{\pm i}$, $i=1,2$, 
$$S(x,d) \approx  \frac{N^2}{\left(1- \frac{|x_1|-d}{N+1}\right)^2\left(1-\frac{|x_2|-d}{N+1}\right)^2} $$
and
$$P_U(x,y)\approx 
 \frac{\left(1-\frac{|x_1|}{N+1}\right) \left(1-\frac{|x_2|}{N+1}\right)  \left(1-\frac{|y_1|-1}{N+1}\right) \left(1-\frac{|y_2|-1}{N+1}\right)    \log (1+|x|)}{ \left(1- \frac{|x_1|-|x-y|}{N+1}\right)^2\left(1-\frac{|x_2|-|x-y|}{N+1}\right)^2 
\log (1+N)   }  $$

\subsection{Conclusion}  
For  reversible Markov chains  killed at the boundary of a finite subdomain $U$, the Doob-transform technique reduces estimates of  the Poisson kernel (harmonic measure) and its time-dependent versions to
estimates of a reversible ergodic (except perhaps for periodicity) Markov chain, where the estimates are determined explicitly in terms of the Perron-Frobenius eigenfunction $\phi_0$.  In general, neither the Perron-Frobenius eigenfunction nor the
resulting ergodic Markov chain are easily studied.  However, when the original Markov chain (or, equivalently, its underlying graph) satisfies a parabolic Harnack inequality, uniformly at all locations and scales, and the finite domain $U$
is an inner-uniform domain, it become possible to reduce all estimates solely to a good understanding of the Perron-Frobenius eigenfunction $\phi_0$. See, Theorems~\ref{theo-PUo} and~\ref{theo-HMarb}. When the finite domain $U$ has a reasonably simple geometry, a variety of relatively sophisticated tools are available to determine the behavior of $\phi_0$ and this leads to sharp two-sided estimates for the Poisson kernel and its time dependent variants.  

In many cases of interest, global estimates  of the Perron-Frobenius eigenfunction $\phi_0$ remain a difficult challenge. The results proved here provide further justifications for attempting to tackle this challenge.  The gambler's ruin problem with four (or more)  players  is a good example of such a  problem. It is amenable to the techniques developed above and it is possible to show that the function $\phi_0$ vanishes in a manner similar to 
different power functions near distinct parts of the boundary. In this and other similar examples, computing the various exponents and putting together
these bits of information to get a global two-sided estimate of $\phi_0$ is a challenging problem.

\section*{Acknowledgements}
Laurent Saloff-Coste was partially supported by NSF grant DMS-1707589. Kelsey Houston-Edwards was partially supported by NSF grants DMS-0739164 and DMS-1645643.

\bibliographystyle{plain}

\bibliography{DHS2}

\end{document}